\documentclass[12pt]{article}
    \usepackage{latexsym}
   \usepackage{amsmath,amsthm,amssymb}
   \usepackage{graphicx}
    \usepackage{here}
    \usepackage{comment}
    \usepackage{color}
    \usepackage{pgfplots}

    \numberwithin{equation}{section}
    \newtheorem{dfn}{Definition}[section]
    \newtheorem{thm}{Theorem}[section]
   \newtheorem{lem}{Lemma}[section]
   \newtheorem{rmk}{Remark}[section]
   \newtheorem{prop}{Proposition}[section]

\begin{document}
\begin{center} 
{\Large \bf A macro-micro elasticity-diffusion system modeling absorption-induced swelling in rubber foams - Proof of the  strong solvability}

\vskip 12pt
Toyohiko Aiki \\
Department of Mathematical and Physical Sciences, Japan Women's University \\
2-8-1 Mejirodai, Bunkyo-ku, Tokyo~112-8681, Japan, \\
\& Karlstad University, Sweden \\
{aikit@fc.jwu.ac.jp}

\vskip 12pt
Nils Hendrik Kr\"oger \\
Deutsches Institut f\"ur Kautschuktechnologie e.\,V. \\(German Institute of Rubber Technology e.\,V.)\\
Eupener Stra\ss e 33, 30519 Hannover, Germany \\
\& material prediction GmbH, Nordkamp 24, 26203 Wardenburg, Germany\\
nils.kroeger@dikautschuk.de, n.kroeger@materialprediction.de

\vskip 12pt
Adrian Muntean \\ 
 Department of Mathematics and Computer Science, Karlstad University\\
 Universitetsgatan 2, 65188, Karlstad, Sweden \\
adrian.muntean@kau.se

\end{center}

{\bf Abstract. } 
In this article,  we propose a macro-micro (two-scale) mathematical model for describing the macroscopic swelling of a rubber foam caused by the microscopic absorption of some liquid. 
In our modeling approach, we suppose that the material occupies a one-dimensional domain which swells as described by the standard beam equation including an additional term determined by the liquid pressure. As special feature of our model,  the absorption takes place inside the rubber foam via a lower length scale, which is assumed to be inherently present in such a structured material. The liquid's absorption and transport inside the material is modeled by means of a nonlinear parabolic equation derived from Darcy's law posed in a non-cylindrical domain defined by the macroscopic deformation (which is a solution of the beam equation). 

Under suitable assumptions,  we establish the existence and uniqueness of a suitable class of solutions to our evolution system  coupling the nonlinear parabolic equation posed on the microscopic non-cylindrical domain with the beam equation posed on the macroscopic cylindrical domain. In order to guarantee the regularity of the non-cylindrical domain, we impose a singularity to the elastic response function appearing in the beam equation.

\section{Modeling background and aim of this paper}

Polymeric materials play a pivotal role in many technological applications. Being hyperelastic or respectively highly viscoelastic, e.g. in case of carbon or silica filled rubbers, those materials are used for damping, coupling or sealing. In order to reduce closing forces or for acoustic reasons instead of full matrix materials foams are used. In their applications rubbers and their foams are often exposed to different kind of chemicals such that storage and transport phenomena through reactive polymeric matrices have both received considerable attention by theorists and experimentalists; see e.g. \cite{Giese, Chester, Guo, Nagdi}. From the theory side, it turns out that, due to the complex internal structure of the involved materials, standard macroscopic laws (like Fickian diffusion) do not hold anymore. Also, it is not clear cut how to predict the mechanical response of the material if its internal microstructures are undergoing significant changes for instance due to microscopic ingress of aggressive chemicals. Such an example is the case of the durability of rubber components put in contact to a large time exposure to ions attack (e.g.\,chlorides from a marine environment).  From the experimental side, the lack of reliable rheological models has as direct consequence that the needed constitutive laws have to be built in the laboratory for each material separately. This way the designing of materials with functional properties becomes possible, but it is yet unclear how the designed materials will respond e.g.\,to environmental conditions (far away from what happened in the laboratory) \cite{blowing}. This context offers us an adequate playground to develop well-posed multiscale models that are able to approximate the macroscopic response of internal (microscopic) changes.

Liquid uptake of polymers is often linked to swelling. The swelling process propagates jointly with diffusion and causes local stresses, which, in return, change locally the propagation of diffusion. Hence, swelling occurs when the mechanics of the material couples with mass transport and {\em viceversa }\cite{Neff_ZAMM}.   In this framework, we develop a new macro-micro evolution system able to describe the simultaneous microscopic absorption of a liquid  that is ultimately responsible for creating a macroscopic swelling of the material. Such model will need to be confronted at a later stage against experimental evidence. 
As starting picture, we assume that the material (rubber foam) is completely soaked in some liquid occupying in a container $\Omega  \subset {\mathbb R}^d$.  
 To keep things simple, we neglect  for now any capillary phenomena; mind that capillarity contributes essentially to the transport of liquid inside the polymer matrix cf.\,e.g.\,\cite{Hewitt}.
  
The rubber foam is a porous material. We imagine it here that it has a clear dual-porosity structure, i.e.\,there exist a finite number of large spherical pores $\Omega_i \subset \Omega$, $i = 1, 2, \cdots, N$ (see Figure \ref{fig1}), as well as many much smaller pores which are taken as inactive with respect to absorption at least at the time scale we are considering our setting. These very small pores are contributing to the set  
 $\Omega_0 = \Omega \setminus \overline{ \cup_{i=1}^N \Omega_i}$, where $\Omega_0$ indicates the region occupied by the rubber foam\footnote{{Alternatively, both $\Omega$ and $\Omega_0$ could be seen as rubber foams with very different local porosities. In such a case, the porosity corresponding to $\Omega$ is significantly smaller than the one for $\Omega_0$; this is some sort of situation with a "high contrast" in porosities.  In this paper, we stick however with the first interpretation of the meaning of $\Omega$ and $\Omega_0$.}}. %\NHK{Please note that we call $\Omega$ as well as $\Omega_0$ rubber foam. From the macro and micro perspective they do not have necessarily the same porosity.}
 
 Benefitting from the dual porosity structure of our material, we consider the competition between the mechanics of the material and the transport of mass as taken place from the perspective of two well-separated space scales, i.e. one {\em macro} and one {\em micro}\footnote{Notice that in the vast majority of  modeling approaches, the fight between diffusion and mechanics takes place at the same observable (macroscopic) scale; see, for instance,  \cite{Chester} and references cited therein.}. The reader is kindly referred to \cite{Ralph} for more information on this modeling philosophy with distributed microstructures. To derive our macro-micro elasticity-diffusion model, we need a number of additional assumptions concerning the involved geometry:
%Here, we give the following list to provide assumptions concerning the domain. 

%\begin{figure}[H]
%\begin{minipage}{0.5\hsize}
%\centering
%\def\svgwidth{200pt}
%\input{Figure1.pdf_tex}
%\caption{Schematic representation of the investigated rubber foam $\Omega$ consisting of a compact foam matrix $\Omega_0$ and of the large pores $\Omega_i$ with connecting boundaries $\Gamma_i$.}
%\label{fig1}
%\end{minipage}
%\quad
%\begin{minipage}{0.5\hsize}
%\centering
%\def\svgwidth{200pt}
%\input{Figure2.pdf_tex}
%\caption{Sketch of the micro model responsible with the liquid transport (wavy blue); liquid accumulation takes place in the large pores.}
%\label{fig3}
%\end{minipage}
%\end{figure}

\begin{figure}[H]
\begin{minipage}{0.5\hsize}
\centering
\includegraphics[width=0.95\textwidth]{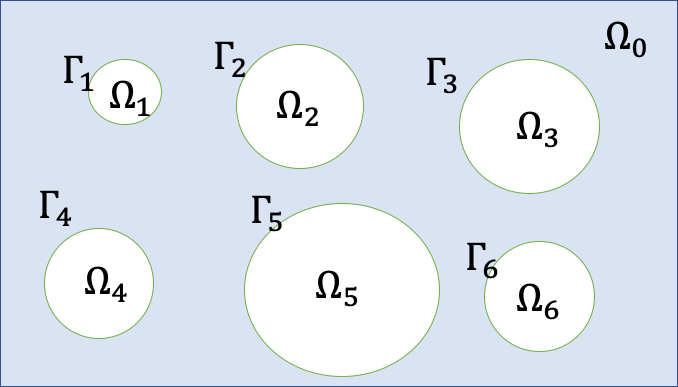}
\caption{Schematic representation of the investigated rubber foam $\Omega$ consisting of a compact foam matrix $\Omega_0$ and of the large pores $\Omega_i$ with connecting boundaries $\Gamma_i$.}
\label{fig1}
\end{minipage}
\quad
\begin{minipage}{0.5\hsize}
\centering
\includegraphics[width=0.85\textwidth]{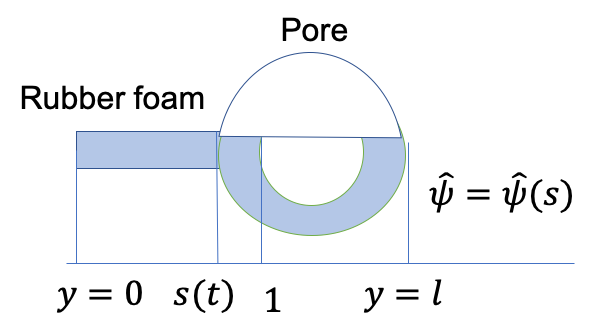}
\caption{Sketch of the micro model responsible with the liquid transport; liquid accumulation takes place in the large pores.}
\label{fig3}
\end{minipage}
\end{figure}

\begin{itemize}
\item $\Omega_i$ depends on time $t$, namely, we write $\Omega_i := \Omega_i(t)$ for $i = 0, 1, 2, \cdots, N$. 
\item The deformation of $\Omega_0$ is represented by the deformation $u$. 
To fix ideas, we consider our material to be viscoelastic; the deformation $u$ satisfies the kinetic equation with viscoelasticity. Also, the main variable in our model is the liquid density $w$ in  $\Omega_0$. Since our material is porous, we can write $w = \rho(p)$ by using  the liquid pressure $p$, where  $\rho$ is a continuous and increasing function on ${\mathbb R}$.
\item  In $\Omega_i$ (for each $i$) the liquid  only accumulates, it never diffuses out. 
\item The change of volume of  the liquid inside $\Omega_i$ determines swelling. Accordingly, the boundary $\Gamma_i (t):= \partial\Omega_i(t)$ is determined by the {stress $\sigma$ or strain $\varepsilon$}   on $\Gamma_i$ and  by volume of the liquid in $\Omega_i$. 
\end{itemize}

In addition to these physical assumptions, for mathematical simplicity we suppose that all involved spacial domains (i.e. $\Omega_i$ and $\Omega_0$) are one-dimensional intervals. Furthermore,  we deal with only the case that one pore (micro) is connected at one side to the rubber foam (macro), see Figure \ref{fig3}.

Firstly, we characterize our rubber foam as a material having good viscoelastic properties; see e.g. \cite{Nagdi}.  Hence, the stress $\sigma$ is given by a modified viscoelastic approach, that is
$$\sigma = f_0( \varepsilon) + k_v \partial_t \varepsilon,$$ where $\varepsilon$ is the strain, $f_0$ is a function on $\mathbb R$ corresponding to the 
{elastic response function},  
%% [COMMENT: the notation stress function might be misleading, since the total stress is divided in elastic part and viscous part.]}, 
and $k_v$ is the effective viscoelastic constant.   Also, since we assumed the existence of a contribution of the liquid pressure to the stress,  we impose
that the total stress $\sigma_{\mbox{total}}$ depends on the liquid  pressure $\hat{p}$ as follows: 
$$ \sigma_{\mbox{total}} = \sigma + \nu(\hat{p}), $$ 
where $\nu$ is a continuous function on  $\mathbb R$. Moreover, we assume that the natural length of the rubber foam is $1$, and 
denote  by $u = u(t,x)$ the position of $x \in (0,1)$ at time $t$ (see Figure \ref{fig2}).   Accordingly, we have $\varepsilon(t,x) = \partial_x u(t,x) - 1$. 
We note that the case $\varepsilon = -1$ means that different points in the material  overlap. 
Since the mass conservation of the liquid is posed on the non-cylindrical domain determined by $u$, as mentioned later,  we need  the existence of the inverse mapping of 
$u(t)$ for each $t$. Here, we propose {an elastic response function} having singularity at $\varepsilon = -1$ as  follows: 
$$ f_0(\varepsilon) = \frac{k}{2}\left( \varepsilon + \frac{1}{2} - \frac{1}{2(1 + \varepsilon)^3}\right), $$ 
where $k$ is the effective elastic constant. From the definition of $f_0$, we see that the magnitude of the stress tends to infinity when the strain goes to $-1$ (see Figure \ref{fig_stress}).  Deviating in our approach from classical linear visco-elasticity, we consider that this assumption is natural from the physical point of view as a huge stress must activate to avoid the overlapping of the different points. 
Actually, in this article we succeed to get uniform estimates for $\varepsilon$ from below (see Lemma \ref{key_lemma}), namely, {$\varepsilon$ is never $-1$ } in our model.  
This kind of stress function {elastic response function} was already treated in Aiki-Kosugi \cite{Aiki-Kosugi} to handle the dynamics of elastic materials by means of large systems of coupled ordinary differential equations.

\begin{figure}[H]
\begin{minipage}{0.5\hsize}
\centering
\includegraphics[width=0.9\textwidth]{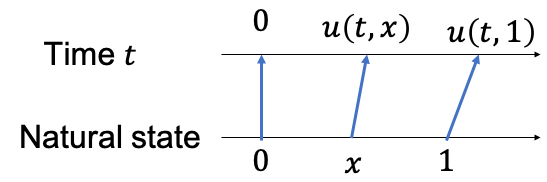}
\caption{Sketch of the 1d domain involved in the micro model.}
\label{fig2}
\end{minipage}
\quad
\begin{minipage}{0.5\hsize}
\centering
\includegraphics[width=0.75\textwidth]{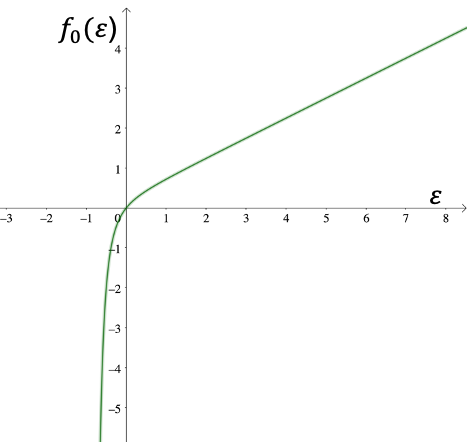}
\caption{Modified elastic response function with $k = 1$. }
\label{fig_stress}
\end{minipage}
\end{figure}

Moreover, we adapt the following beam equation as the kinetic equation for the rubber foam: 
\begin{equation}
m (\partial_{t})^2 u  + \gamma (\partial_x)^4 u - \partial_x \sigma -  \partial_x \nu (\hat{p}) = 0 \mbox{ in } Q(T) := (0,T) \times (0,1),  \nonumber %%\label{kinetic_eq}
\end{equation} 
where $m >0$ is the density of the foam and $\gamma$ is a  positive constant. This type of partial differential equations was already investigated as a mathematical model for thin elastic materials, compare, for instance, the monograph by  Brokate and Sprekels \cite{Bro-Sp}.  
%%When $\gamma = 0$, \eqref{kinetic_eq} is corresponding to Newton's law.  
{Concerning}  the boundary condition for $u$ at $x = 0$,  we impose that one edge $x = 0$ is fixed, namely, $u(t,0) = 0$. It is clear that the length  $s(t)$ of the rubber foam  at time $t$ is given by $u(t,1)$, i.e. $s(t) = u(t,1)$ for $t \in [0,T]$.

Next, we consider the mass conservation law for the liquid present inside the rubber foam  and occupying the region $(0, s(t))$ for each $t \in (0, T)$. We denote the density and the liquid pressure in the porous material, respectively, by $w = w(t, y)$ and $p = p(t, y)$ for $(t, y)  \in  Q(s, T) := \{(t,y)  : 0< t < T, 0 < y < s(t)\}$. Here, we note that 
$w = \rho(p)$ as mentioned above and $\hat{p}(t, x) = p(t, u(t, x))$ for $(t, x)  \in Q(T)$, namely,  $\hat{p}(t, u^{-1}(t, y)) = p(t, y)$ for $(t, y) \in Q(s, T)$, where $u^{-1}(t,y)$ denotes the inverse of $y =u(t, x)$ for each $t \in [0,T]$. Thus, we remark that the inverse of the mapping $y = u(t, x)$ is necessary for analysis to this problem.  By Darcy's law,  the flow is given by $ -\kappa \partial_y p$, where $\kappa$ is the permeability which is a positive constant. Also, since there exists advection in this problem, the mass conservation law is written by
$$  \partial_t \rho(p) + \partial_y(\hat{v} \rho(p)) - \kappa (\partial_y)^2 p = 0 \mbox{ in } Q(s, T), $$
where the velocity $\hat{v}(t, y) := v(t,  u^{-1}(t, y))$ for $(t, y) \in Q(s, T)$ and  $v(t,x) = \partial_t u(t, x)$ for $(t, x) \in Q(T)$.\\

We give a remark concerned with domains of some variables appearing in the model.  We consider the kinetic equation (resp. mass conservation) on the cylindrical domain $Q(T)$  (resp. the non-cylindrical domain $Q(s, T)$) and $x \in (0,1)$  (resp. $y \in (0, s(t))$) is used as the argument of the functions in this article.

At $y = 0$  we impose the following non-homogenous Neumann boundary condition  for $p$ in order to represent the change of the liquid volume:
$$  \kappa \partial_y p(t,0) = h_0(t) \quad \mbox{ for } 0 < t < T, $$
where $h_0$ is a given function on $[0,T]$. 

In order to describe liquid accumulation in the pore  $(s(t), l)$, where $l$ is a given constant, we suppose that the liquid mass $\hat{\psi} $ in the pore is determined by only $s$, that is, we can  write $\hat{\psi} = \hat{\psi}(s)$ (see Figure \ref{fig3}). For a timespan of size $\Delta t  > 0$, we consider the change of the liquid mass within $(s(t+\Delta t), s(t))$ from  $t$ to $t +\Delta t$.
This yields 
$$ (\mbox{change of the liquid mass})  = (\mbox{flux of pressure}) + (\mbox{flow by motion of the material}). $$ 
In that case,  we have  
\begin{align*}
& \hat{\psi}(s(t+\Delta)) - \hat{\psi}(s(t)) - (s(t) - s(t + \Delta t)) w(t,s(t))  \\
= &  -\kappa \partial_y p(t, s) \Delta t  + \hat{v}(t, s(t+\Delta t)) w(t, s(t+ \Delta)) \Delta t. 
\end{align*} 
Therefore, thanks to $s'(t) = \hat{v}(t, s(t))$ by letting $\Delta t \to 0$, we get
\begin{equation}
 s'(t)  \psi(s(t))  +\kappa   \partial_y p(t, s(t)) = 0. \label{FBC}
 \end{equation}
We confirm that  the conservation law \eqref{fbp_mass} of the liquid mass in the whole domain holds under  \eqref{FBC}, when $h_0 \equiv 0$:  
\begin{equation}
\frac{d}{dt} \left( \int_0^{s(t)} w(t, y)dy + \hat{\psi}(s(t)) \right) =0 \quad \mbox{ for } t \in (0,T).  \label{fbp_mass}
\end{equation}
Indeed,  we observe  that 
\begin{align*}
& \frac{d}{dt} \int_0^{s(t)} w(t, y)dy   \\
=  & s'(t) w(t, s(t)) + \int_0^{s(t)} \partial_t w(t, y) dy \\
=  &s'(t) w(t, s(t)) + \int_0^{s(t)} ( \kappa (\partial_y)^2 p(t,y) - \partial_y (\hat{v} w) )dy  \\
 = & s'(t) w(t, s(t)) + \kappa (\partial_y p(t, s(t))  - \partial_y p(t, 0))   -  \hat{v}(t, s(t)) w(t, s(t)) + \hat{v}(t, 0)w(t,0)  \\
 = &  \kappa \partial_y p(t, s(t)) \quad \mbox{ for } t \in (0,T),  
\end{align*}
since $s'(t) = \partial_t u(t,1) = \hat{v}(t, s(t))$ and $\hat{v}(t,0) = \partial_t u(t,0) = 0$.  
Hence, the liquid mass is conserved.% in our model. 

Moreover, at $x = s(t)$ the internal {stress} $\sigma_{\mbox{total}}$  should balance  the force  caused by swelling. We assume that the force 
depends only on $s$, namely, the force is presented by $\varphi(s)$,  where $\varphi$ is a continuous function on ${\mathbb R}$. 
Thus,  we obtain the following boundary condition: 
$$ - \gamma (\partial_x)^3 u(t,1) + \sigma(t,1) + \varphi( s(t) ) + \nu(\hat{p}) (t, s(t))  = 0   \quad \mbox{ for } t \in [0,T]. $$

Summarizing the above discussion, we obtain the following system:  Find the functions functions 
$u$ and $p$ satisfying the following model equations:
\begin{align}
& m u_{tt}  + \gamma u_{xxxx} -  ( f(u_x) + k_v u_{xt})_x - \nu (\hat{p})_x = 0 \mbox{ in } Q(T),  \label{EQ1}  \\
& \mbox{ where } f(u_x) = \frac{k}{2}( u_x - \frac{1}{2} - \frac{1}{2(u_x)^3}),   \nonumber \\
& u(t, 0) = 0,  u(t,1) = s(t), u_{xx}(t,0) = u_{xx}(t,1) = 0 \mbox{ for } t \in [0,T], \label{BC1-1}  \\
&  - \gamma u_{xxx}(t,1)+  f(u_x)(t,1) + k_v u_{xt}(t,1) + \nu(\hat{p})(t, 1) + \varphi( s(t)) =0 \mbox{ for } t \in [0,T], \label{BC1-2} \\
&  \rho(p)_t + (\hat{v}  \rho(p))_y - \kappa  p_{yy} = 0  \mbox{ in } Q(s, T),  \label{EQ2}\\
& \kappa  p_y(t,0) = h_0(t), \kappa  p_y(t, s(t)) =  - s'(t) \psi(s(t))    \mbox{ for } 0 < t < T, \label{BC2} \\
& u(t,0) = u_0(x), u_t(0, x) = v_0(x)   \mbox{ for } x \in (0,1), \label{IC1} \\
& p(0,y) = p_0(y) \mbox{ for } y \in (0,s(0)), \label{IC2} 
\end{align}
where $\hat{p}(t, x)  = p(t, s(t)x)$ for  $(t,x) \in Q(T)$, 
$u_0$ and $p_0$ are initial functions, and $\psi = V'$. 

We refer to (\ref{EQ1})--(\ref{IC2}) as problem (P). Understanding the behavior of solutions to this problem becomes now our main target.

\section{Concept of solutions to problem (P) and an uniqueness result}

In this section, we introduce a suitable concept of solutions to problem (P).  For simplicity, we denote the following function spaces as follows: 
$$ H = L^2(0,1), X = \{z \in W^{2,2}(0,1) : z(0) = 0\}, V = H^{4}(0,1) \cap X. $$
Also, we denote the usual inner products of $H$ and $V$ by $(\cdot, \cdot)_H$ and $(\cdot, \cdot)_V$, respectively.  

\begin{dfn} \label{def_solution}  \rm
Let $u$ be a function on $Q(T)$ and $s(t) = u(t,1)$ for $t \in [0,T]$. Also, let $p$ be a function on $Q(s, T)$. 
We call that the pair $(u, p)$  a solution to (P) on $[0, T]$, if the following conditions (S1)-(S4) hold: 

\begin{itemize}
\item[(S1)]  $u \in W^{2,\infty}(0,T; H) \cap W^{1,\infty}(0,T; X) \cap W^{2,2}(0,T; H^1(0,1)) \cap L^{\infty}(0, T;  V) =: S_1(T)$. 
\item[(S2)] $u_x > 0$ on $\overline{Q(T)}$. 
\item[(S3)] $p_{yy}, p_{t} \in L^2(Q(s, T))$, $|p(\cdot)|_{H^1(0,s(\cdot))} \in L^{\infty}(0,T)$. 
\item[(S4)]  \eqref{EQ1} - \eqref{IC2} hold in the usual sense. 
\end{itemize}

\end{dfn}

\begin{rmk} \label{Remark1}  \rm
Obviously, if $u$ satisfies (S1)  and (S2), then $u_x \in C(\overline{Q(T))}$ and   $u^{-1}(t, \cdot)$ exists for each $t \in [0, T]$ and is continuous on $C(\overline{Q(T))}$. 
Furthermore, we can take positive constants $\delta$ and $M$  such that $\delta  \leq u_x \leq M$ on $\overline{Q(T)}$.  
This shows that $\delta \leq s(t) \leq M$ for $t \in [0, T]$. Thus,  as they appear in \eqref{EQ2}, the objects $Q(s, T)$ and $\hat{v}$ are well-defined. 
\end{rmk} 

Next, we list assumption for given data. 
\begin{itemize}
\item[(A1)]$\rho \in C^1({\mathbb R})$ be a Lipschitz continuous and  increasing function on $\mathbb{R}$ with having Lipschitz continuous inverse $\beta = \rho^{-1}$. Immediately, we see that $\mu := \inf\{\rho'(r) : r \in {\mathbb R}\} > 0$.  Moreover, we suppose that 
$\beta'$  is also Lipschitz continuous.  For simplicity, we put
$C_{\rho} =  |\rho'|_{L^{\infty}({\mathbb R})} + |\beta'|_{W^{1,\infty}({\mathbb R})}$. 
\item[(A2)] $\varphi, \psi \in C^1({\mathbb R}) \cap W^{1,\infty}({\mathbb R})$. We denote their primitives by 
$\hat{\varphi}$ and $\hat{\psi}$, respectively, and put $C_{\varphi} = |\varphi|_{W^{1,\infty}({\mathbb R})}$ and  $C_{\psi} = |\psi|_{W^{1,\infty}({\mathbb R})}$.
\item[(A3)] $\nu$ is a Lipschitz continuous and bounded function on ${\mathbb R}$, and we put $C_{\nu} = |\nu|_{W^{1,\infty}({\mathbb R})}$.  
\end{itemize}

\vskip 12pt
The main result is concerned with existence and uniqueness of a solution to (P). 

\begin{thm} \label{main_th}
 If (A1) - (A3) hold, $u_0 \in V$ with $u_{0x} > 0$ on $[0,1]$,  {$u_{0xx}(0) = u_{0xx}(1) = 0$}, $v_0 \in X$, 
$p_0 \in W^{1,2}(0,s_0)$ and { $- \gamma u_{0xxx} (1) + f(u_{0x}(1)) + k_v v_{0x}(0) + \nu( p_0(s_0)) +\varphi(s_0) = 0$, 
}
where $s_0 = u_0(1)$, then the problem (P) has a unique solution  on $[0, T]$. 
\end{thm}

\vskip 12pt
In Theorem \ref{main_th} by imposing  boundedness for the functions $\nu$ and $\varphi$  we can obtain the uniform estimate for $u$ and then prove the existence of solutions. Here, we show a way to get the estimates, briefly. Let $u$ be a function satisfying \eqref{EQ1}. 
By multiplying \eqref{EQ1} with  $u_t$ and integrating it we have the following inequality: 
\begin{align*}
 & \frac{m}{2} |u_t(t)|_H^2 +  \frac{\gamma}{2} |u_{xx}(t)|_H^2 +  \frac{k}{2} |u_{x}(t)|_H^2 
+ \int_ 0^1 \frac{k}{8 |u_x(t)|^2} dx + \hat{\varphi} (u(t,1)) + k_v \int_0^t |u_{tx}(\tau)|_H^2 d\tau \\ 
 \leq &  \frac{m}{2} |v_0|_H^2 +  \frac{\gamma}{2} |u_{0xx}|_H^2 +  \frac{k}{2} |u_{0x}|_H^2 
+  \int_ 0^1 \frac{k}{ 8|u_{0x}|^2 } dx + \int_0^1 u_x(t) dx + \hat{\varphi}(u_0(1)) - \int_0^1 u_{0x} dx   \\
&  - \int_0^t \int_0^1 \nu(\hat{p})(\tau)  u_{tx}(\tau) dx d\tau
      \mbox{ for } t \in [0,T]. 
 \end{align*}
Thanks to the boundedness of $\nu$ and $\varphi$, the uniform estimates of $u$ is easily obtained. We note that the estimate is independent of $\hat{p}$. 
Moreover, according to Lemma \ref{key_lemma} in this paper and the assumption $u_{0x} > 0$ on $[0,1]$, it holds that  
$u_x > 0$ in  $\overline{Q(T)}$. Based on these estimates, we can prove the theorem.\\

In this paper, we define the stress function {elastic response function} by {a specific form. } We  comment in Remark \ref{rem_gen} on eventual generalizations of $f$.\\

For proving Theorem \ref{main_th}, we shall solve two distinct problems, viz.  
\begin{itemize}
\item[(P1)] ($\hat{p}, u_0,  v_0) := \{\eqref{EQ1}, \eqref{BC1-1},\eqref{BC1-2},  \eqref{IC1}\}$, and
\item[(P2)] ($\hat{v}, s, h_0, p_0) := \{\eqref{EQ2}, \eqref{BC2}, \eqref{IC2} \}$,
\end{itemize}
for given $\hat{p}$ and $\hat{v}$, respectively. \\

Furthermore, in order to deal with (P2)($\hat{v}, s, h_0, p_0$), we introduce the following notation and change of variable: 
\begin{equation}
\bar{p}(t, x) := p(t, s(t)x) \quad \mbox{ for } (t, x) \in Q(T).  
\end{equation}
By using this definition we have the following problem ($\overline{\mbox{P2}})(\bar{v}, s, h_0, \bar{p}_0$): 
\begin{align}
&  \rho(\bar{p})_t  -  \frac{\kappa}{s^2}  \bar{p}_{xx} = - \frac{1}{s}  (\bar{v}  \rho(\bar{p}))_x  + \frac{s'}{s} x \rho(\bar{p})_x  \mbox{ in } Q(T),  \label{EQ2'}\\
& \frac{\kappa}{s(t)}   \bar{p}_x(t,0) = h_0(t), \frac{\kappa}{s(t)}  \bar{p}_x(t, 1) =  - s'(t) \psi(s(t))    \mbox{ for } 0 < t < T, \label{BC2'} \\
& \bar{p}(0,y) = \bar{p}_0(x) \mbox{ for }  x \in (0,1), \label{IC2'} 
\end{align} 
where $\bar{p}_0(x) = p_0(s_0 x)$ for $x \in (0,1)$.\\

This paper is organized as follows: 
In Section \ref{aux_lemmas} we give some useful inequalities as lemmas.  In Section \ref{unique} we prove uniqueness of solutions to (P). 
Since we can regard (P1) and (P2) as special cases of (P), the uniqueness for (P) implies uniqueness of solutions for (P1) and (P2), immediately. 
Next, for given $\hat{p}$ we solve  (P1)  in Section \ref{P1}  by applying the Galerkin method. Also, by standard techniques for nonlinear parabolic equations we can show solvability  of (P2) in Section \ref{P2}. Finally, at the end of the paper, we establish the existence of solutions to (P).

\section{Auxilliary lemmas} \label{aux_lemmas} 
Our proofs rely on a couple of key inequalities that we like to present in this section. 
Firstly, let us recall the Gagliardo-Nirenberg interpolation inequality. 
\begin{lem} \label{lemma1}
For any $z \in H^1(0,1)$ the following inequalities holds:  
\begin{align*}
& |z|_{L^{\infty}(0,1)}  \leq 2 (|z_x|_H^{1/2} |z|_H^{1/2} + |z|_H) \leq 3( |z_x|_H +  |z|_H), \\
 & |z|_{L^{\infty}(0,1)}  \leq 2 |z_x|_H^{1/2} |z|_H^{1/2}, \mbox{ if } z(0) = 0. 
 \end{align*}
\end{lem} 

\begin{lem} \label{key_lemma}
 Let $z \in X$,  and  $r_1$ and $r_2$ be any positive constants. 
If 
\begin{equation}  
\int_0^1 \frac{1}{|z_x|^2} dx \leq r_1, |z|_X \leq r_2,  \label{assump_key}
\end{equation}
then it holds 
$$ |z_x| \geq \frac{r_2}{\sqrt{2}} e^{- r_1 r_2^2} \mbox{ on } [0,1]. $$
\end{lem} 

{
\begin{proof}
Assume \eqref{assump_key}. We have
$$ |z_x(x) - z_x(x')| \leq |z_{xx}|_H |x - x'|^{1/2} \leq r_2 |x - x'|^{1/2} \mbox{ for } x, x' \in [0,1]. $$ 
For $x_0 \in [\frac{1}{2}, 1]$ 
 we observe that  
\begin{align*}
\int_0^1 \frac{1}{|z_x(x)|^2} dx & \geq \frac{1}{2} \int_0^{x_0} \frac{1}{|z_x(x) - z_x(x_0)|^2 + |z_x(x_0)|^2} dx    \\
& \geq \frac{1}{2} \int_0^{x_0} \frac{1}{r_2 (x_0 - x) + |z_x(x_0)|^2} dx \\
& \geq \frac{1}{2K_2^2} \log \left( \frac{r_2^2 x_0 + |z_x(x_0)|^2}{|z_x(x_0)|^2} \right) \\
& \geq \frac{1}{2K_2^2} \log \left( \frac{r_2^2 }{2 |z_x(x_0)|^2} \right).  
\end{align*}
Accordingly, we get 
$$  |z_x(x_0)| \geq \frac{r_2}{\sqrt{2}} e^{- r_1 r_2^2} \mbox{ for }  x_0 \in [\frac{1}{2}, 1]. $$
Similarly, we can show that  this inequality holds  for $\displaystyle x_0 \in [0, \frac{1}{2}]$. 
\end{proof}
}

{As mentioned before, generalizations for the stress function elastic response function $f$ are possible as discussed in the following remark. }
\begin{rmk} \label{rem_gen} \rm
Suppose that the stress function {elastic response function} $f$ is defined  on $(0, \infty)$ and its primitive $\hat{f}$  satisfies 
\begin{equation}
\hat{f}(r) \geq \frac{c}{ r^q} \mbox{ for } r > 0,   
\label{cond_f}
\end{equation}
where $c$ and $q$ are positive constants. 
As easily seen, if $q \geq 2$,   then a similar estimate as in Lemma \ref{key_lemma} holds. 
Thus, since it is possible to generalize this lemma,  we may prove existence and uniqueness for $f$ satisfying \eqref{cond_f}. 
\end{rmk} 

\section{Uniqueness} \label{unique}
The aim of this section is to prove the uniqueness of solutions to (P). Throughout this section we suppose that all assumptions from the hypothesis of Theorem \ref{main_th} hold. Furthermore, we denote by $(u_1, p_1)$ and $(u_2, p_2)$ two sets of solutions to (P) defined on the same timespan $[0,T]$. 
Also,  we put 
\begin{align*} 
& u = u_1 - u_2 \mbox{ on } Q(T),  s_i(t) = u_{i}(t,0)  \mbox{ for }  t \in  [0,T] \mbox{  and  } i = 1, 2,  s = s_1 - s_2 \mbox{ on } [0,T],  \\ 
& \bar{p}_i(t, x ) = p_i(t, s_i(t)x),  \hat{p}_i(t, x ) = p_i(t, u_i(t,x)) \mbox{  for } (t, x) \in Q(T),  i = 1, 2,  \\
& \hat{p} = \hat{p}_1 - \hat{p}_2,  \bar{p} = \bar{p}_1 - \bar{p}_2  \mbox{ on } Q(T),  \\
& \sigma_i = f(u_{ix}) + k_v u_{tx}, i = 1, 2, \sigma = \sigma_1 -  \sigma_2 \mbox{ on } Q(T),  \\
& v_i (t,x)= u_{it}(t,x),  \bar{v}_i(t,x) = v_i(t, u_i^{-1}(t, s_i(t)x)) \mbox{ for } (t,x)  \in Q(T),  i = 1, 2, \\
&  v= v_1 -  v_2, \bar{v} = \bar{v}_1 - \bar{v}_2 \mbox{ on } Q(T),  \\
& w_i = \rho(p_i) \mbox{ on } Q(s_i, T), \bar{w}_i(t,x) = w_i(t, s_i(t)x), \mbox{  for } (t, x) \in Q(T), i= 1,2,  \\
& \bar{w} = \bar{w}_1 - \bar{w}_2 \mbox{  on }  Q(T). 
\end{align*}

By Definition \ref{def_solution} we can take $\delta > 0$, $M > 0$  and $R > 0$ such that for $i = 1, 2$, 
\begin{align}
 \left\{  \begin{array}{l}
 \delta  \leq u_{ix} \leq M \mbox{ on } \overline{Q(T)}, \\
 |u_{itx}|_H \leq R,  |\bar{v}_{i}|_{L^{\infty}(0,1)}  \leq R, |v_{ix}|_{L^{\infty}(0,1)}  \leq R, |\bar{v}_{ix}|_{L^{\infty}(0,1)}  \leq R \mbox{ on }  [0,T],  \\
  |\bar{w}_i|_H \leq R, |\bar{w}_i|_{L^{\infty}(0,1)} \leq R, |\beta(\bar{w}_i)_x|_H \leq R \mbox{ on }   [0,T]. 
 \end{array} \right. 
 \label{4a0}
\end{align}
The role of (\ref{4a0}) is to ensure that the following bounds hold:  
\begin{equation}
\left\{  \begin{array}{l}
|s_i(t)| \leq M,  |s_i'(t)| \leq R \mbox{ for } t \in [0,T], |u_i(t,x)| \leq M  \mbox{ for } (t,x) \in \overline{Q(T)} \mbox{ and } i = 1, 2,  \\
|s(t)| \leq |u_x(t)|_H, |s'(t)| \leq |u_{tx}(t)|_H \mbox{ for } t \in  [0,T]. 
\end{array}  \right.  \label{4a1}
\end{equation}
The following lemma gives an energy-like estimate for $u$: 
\begin{lem} \label{Uni_P1} 
There exists a positive constant $C_1$ depending only on $C_{\varphi}$, $C_{\nu}$, $k$, $k_v$, $\delta$ and $M$ satisfying \eqref{4a1} such that 
\begin{align}
& \frac{d}{dt} ( \frac{m}{2} |u_t(t)|_H^2 + \frac{\gamma}{2} |u_{xx}(t)|_H^2 + \frac{k}{4} |u_x(t)|_H^2) + \frac{k_v}{2}  |u_{tx}(t)|_H^2 \nonumber \\
 \leq & C_1(1 + |\bar{p}_{1x}(t)|_{L^{\infty}(0,1)}^2 ) |u_x(t)|_H^2  + C_1 |\bar{p}(t)|_H^2  \quad \mbox{ for } t \in [0,T].  \label{4-0} 
\end{align}
\end{lem} 
\begin{proof}
First, we have 
\begin{equation}
m u_{tt} + \gamma u_{xxxx} - \sigma_x  -( \nu(\hat{p}_1) - \nu(\hat{p}_2) )_x = 0 \mbox{ a.e. in } Q(T). \label{diff_EQ1} 
\end{equation} 
By multiplying \eqref{diff_EQ1} with $u_t$ and integrating it over $[0,1]$, we see that 
\begin{equation}
m (u_{tt}, u_t)_H  + \gamma (u_{xxxx},u_t)_H - (\sigma_x, u_t)_H  -( (\nu(\hat{p}_1) - \nu(\hat{p}_2) )_x, u_t)_H = 0 \mbox{ a.e. on } [0,T]. \label{4-a} 
\end{equation} 
For simplicity we denote its left hand side by $\sum_{i=1}^4 I_i(t)$ for a.e. $t \in [0,T]$.  Thanks to integration by parts and  \eqref{BC1-1}  
it is easy to see that 
\begin{equation} 
I_1(t) + I_2(t) =  \frac{d}{dt} (\frac{m}{2} |u_t(t)|_H^2  + \frac{\gamma}{2} |u_{xx}(t)|_H^2 )+ \gamma u_{xxx}(t,1) u_t(t,1) \mbox{ for a.e. } t \in [0,T], \label{4-b} 
\end{equation} 
\begin{align} 
I_3(t) + I_4(t)  = &  (\sigma(t), u_{tx}(t))_H +  (\nu(\hat{p}_1)(t) - \nu(\hat{p}_2)(t), u_{tx}(t))_H  \label{4-c}   \\
& - \sigma(t,1) u_t(t,1) -  (\nu(\hat{p}_1)(t,1) - \nu(\hat{p}_2)(t,1)) u_{t}(t,1)  \mbox{ for a.e. } t \in [0,T],  \nonumber 
\end{align} 
\begin{align} 
&  (\sigma(t), u_{tx}(t))_H  \nonumber \\
  = & \frac{k}{2} (u_x(t) - (\frac{1}{2u_{1x}(t)^3} - \frac{1}{2u_{2x}(t)^3}), u_{tx}(t))_H + k_v|u_{tx}(t)|_H^2 \nonumber \\
 =  & \frac{k}{4} \frac{d}{dt} |u_x(t)|_H^2 + \frac{k}{4} (\frac{u_{1x}(t)^3 - u_{2x}(t)^3}{ (u_{1x}(t) u_{2x}(t))^3}, u_{tx}(t))_H  + k_v|u_{tx}(t)|_H^2
 \mbox{ for a.e. } t \in [0,T].  \label{4-d} 
\end{align} 
Combining \eqref{4-a} $\sim$ \eqref{4-d} and \eqref{BC1-2},  it follows that 
\begin{align}
& \frac{d}{dt} ( \frac{m}{2} |u_t(t)|_H^2  + \frac{\gamma}{2} |u_{xx}(t)|_H^2 + \frac{k}{4}  |u_x(t)|_H^2 ) + k_v|u_{tx}(t)|_H^2 \\
= & - ( \varphi(s_1(t)) - \varphi(s_2(t))) u_t(t,1) -  \frac{k}{4} (\frac{u_{1x}(t)^3 - u_{2x}(t)^3}{ (u_{1x}(t) u_{2x}(t))^3}, u_{tx}(t))_H   \label{4-e} 
 \\
&  -  (\nu(\hat{p}_1)(t) - \nu(\hat{p}_2)(t), u_{tx}(t))_H  \nonumber \\
=:  & \sum_{i=1}^3 \hat{I}_i(t) \quad  \mbox{ for a.e. } t \in [0,T].  \nonumber 
\end{align} 
Here, we note that 
\begin{align}
|\hat{I}_1(t)| & \leq  C_{\varphi} |u_t(t,1)| |u(t,1)| \nonumber \\
     & \leq C_{\varphi}  |u_{tx}(t)|_H |u_x(t)|_H     \nonumber  \\
     & \leq \frac{k_v}{4} |u_{tx}(t)|_H^2 + \frac{C_{\varphi}^2}{k_v} |u_x(t)|_H^2 \quad  \mbox{ for a.e. } t \in [0,T],  \label{4-f}  \\
|\hat{I}_2(t)| & \leq  \frac{3kM^2}{\delta^6}  |u_x(t)|_H |u_{tx}(t)|_H \nonumber \\
        & \leq \frac{k_v}{8} |u_{tx}(t)|_H^2 + \frac{2 k^2 M^4}{k_v \delta^{12}} |u_x(t)|_H^2 \quad  \mbox{ for a.e. } t \in [0,T],  \label{4-g} \\
|\hat{I}_3(t)| & \leq   C_{\nu} |\hat{p} (t)|_H |u_{tx}(t)|_H \nonumber \\
        & \leq \frac{k_v}{8} |u_{tx}(t)|_H^2 + \frac{2 C_{\nu}^2}{k_v } |\hat{p}(t)|_H^2 \quad  \mbox{ for a.e. } t \in [0,T].  \label{4-h} 
\end{align}
By substituting \eqref{4-f} $\sim$ \eqref{4-h} into \eqref{4-e} we obtain 
\begin{align}
&  \frac{d}{dt} (\frac{m}{2} |u_t(t)|_H^2  + \frac{\gamma}{2} |u_{xx}(t)|_H^2 + \frac{k}{4}  |u_x(t)|_H^2 )+ \frac{k_v}{2} |u_{tx}(t)|_H^2 \nonumber \\
\leq & (\frac{C_{\varphi}^2}{k_v} + \frac{2 k^2 M^4}{k_v \delta^{12}})  |u_x(t)|_H^2 + \frac{2 C_{\nu}^2}{k_v } |\hat{p}(t)|_H^2 
\quad  \mbox{ for a.e. } t \in [0,T].  \label{4-i} 
\end{align} 

We can estimate $|\hat{p}(t)|_H^2$ by $|\bar{p}(t)|_H^2$. 
Indeed, by the definitions of $\hat{p}$ and $\bar{p}$ we infer that 
\begin{align} 
 |\hat{p}(t)|_H^2  
= &  \int_0^1 |\bar{p}_1(t, \frac{u_1(t,x)}{s_1(t)}) - \bar{p}_2(t, \frac{u_2(t,x)}{s_2(t)})|^2 dx \nonumber  \\
 \leq  &  2  \int_0^1 |\bar{p}_1(t, \frac{u_1(t,x)}{s_1(t)}) - \bar{p}_1(t, \frac{u_2(t,x)}{s_2(t)})|^2 dx 
  + 2  \int_0^1 |\bar{p}_1(t, \frac{u_2(t,x)}{s_2(t)}) - \bar{p}_2(t, \frac{u_2(t,x)}{s_2(t)})|^2 dx  \nonumber  \\
   \leq &   2 |\bar{p}_{1x}(t)|_{L^{\infty}(0,1)}^2  \int_0^1 | \frac{u_1(t,x)}{s_1(t)} - \frac{u_2(t,x)}{s_2(t)}|^2 dx 
  + 2  \int_0^1 |\bar{p}_1(t, \xi) - \bar{p}_2(t, \xi)|^2  \frac{s_2(t)}{u_{2x}(t,\xi)} d\xi \nonumber  \\\nonumber 
 \leq  &   \frac{4}{\delta^2} |\bar{p}_{1x}(t)|_{L^{\infty}(0,1)}^2  \int_0^1 ( |s(t)|^2 |u_1(t,x)|^2 + |s_2(t)|^2 |u(t,x)|^2) dx
   + \frac{2}{\delta}  \int_0^1 |s_2(t)| |\bar{p}(t, \xi)|^2  d\xi \nonumber  \\
    \leq  &   \frac{8M^2}{\delta^2} |\bar{p}_{1x}(t)|_{L^{\infty}(0,1)}^2   |u_x(t)|_H^2   
   + \frac{2M}{\delta}  |\bar{p}(t)|_H^2 \quad \mbox{  for a.e. }  t \in [0,T].  \nonumber 
 \end{align}
We note that   $\bar{p}_1(t, \frac{u_2(t,x)}{s_2(t)})$ is well-defined in the above calculations, since  $0 \leq \frac{u_2(t,x)}{s_2(t)} \leq 1$ for a.e. $(t, x) \in Q(T)$. Thus, there exists a positive constant $C_1$ satisfying \eqref{4-0}, where  $C_1$ depends only on 
$C_{\varphi}$, $C_{\nu}$, $k$, $k_v$, $\delta$ and $M$.   
\end{proof} 

The following lemma is a key  to get  estimates for the difference $\bar{w}$ . 
\begin{lem} \label{bar_v} 
There exists a positive constant $C_2$ depending only on $\delta$, $M$ and $R$ such that 
$$ |\bar{v}(t)|_H \leq C_2 (|u_x(t)|_H  + |u_t(t)|_H) \mbox{ for } t \in [0,T]. $$ 
\end{lem}
\begin{proof} 
By the definition of $\bar{v}$ and change of variable we see that 
\begin{align}
|\bar{v}(t)|_H^2  \leq & 2 \int_0^1 |v_1(t, u_1^{-1}(t, s_1(t)x)) - v_1(t, u_2^{-1}(t, s_2(t)x))|^2 dx \nonumber \\
    &     + 2 \int_0^1 |v_1(t, u_2^{-1}(t, s_2(t)x)) - v_2(t, u_2^{-1}(t, s_2(t)x))|^2  dx \nonumber \\
 \leq &   2  |v_{1x}(t)|^2_{L^{\infty}(0,1)} \int_0^1 |u_1^{-1}(t, s_1(t)x) - u_2^{-1}(t, s_2(t)x)|^2 dx \label{40-1} \\
    &     + 2 \int_0^1 |v_1(t, \xi) - v_2(t, \xi)|^2 \frac{u_{2x}(t, \xi)}{s_2(t)}  d\xi  (=: I_1(t) + I_2(t)) \mbox{ for } t \in [0,T].  \nonumber
\end{align}
Since we extend $u_1^{-1}(t, \cdot): [0, s_1(t)] \to [0,1]$ to the function on $(0, M)$ by $u_1^{-1}(t,y) = u_1^{-1}(t, s_1(t)) = 1$ for $ y \in (s_1(t), M)$
and $t \in [0,T]$, 
we can calculate $I_1(t)$ in the following way: 
\begin{align*}
& \  |u_1^{-1}(t, s_1(t)x) - u_2^{-1}(t, s_2(t)x)|  \\
\leq  & \ |u_1^{-1}(t, s_1(t)x) - u_1^{-1}(t, s_2(t)x)|  +  |u_1^{-1}(t, s_2(t)x) - u_2^{-1}(t, s_2(t)x)|  \\
=:  &  \ I_{11}(t,x) + I_{12}(t,x)   \mbox{ for } (t,x) \in Q(T). 
\end{align*} 
This yields 
\begin{align}
 I_{11}(t, x)  \leq |u_{1x}^{-1}(t)|_{L^{\infty}(0,M)} |s_1(t)x - s_2(t)x|  \leq \frac{1}{\delta} |u_x(t)|_H \quad \mbox{ for } (t,x) \in Q(T). 
  \label{40-1.5}
\end{align} 
For $I_{12}$ it holds that 
\begin{equation}
\begin{array}{ll}
I_{12}(t,x)  \leq  & \displaystyle \frac{1}{\delta} |u_1(t, u_2^{-1}(t, s_2(t)x)) - u_2(t, u_2^{-1}(t, s_2(t)x))| \mbox{ for } (t,x) \in Q(T).  
\end{array}
 \label{40-2} 
 \end{equation} 
 In fact, we put $u_1^{-1}(t, s_2(t)x) = U_1$ and $u_2^{-1}(t, s_2(t)x) = U_2 \in [0,1]$. 
 It is clear that $s_2(t)x =u_2(t, U_2)$. If $s_2(t)x \leq s_1(t)$, then $s_2(t)x = u_1(t, U_1)$. 
 This implies that 
 \begin{align*}
& 0 =  u_1(t, U_1) - u_2(t, U_2) = u_1(t, U_1) - u_1(t, U_2)  + u_1(t, U_2) - u_2(t, U_2), & \\  
&  |u_1(t, U_2) - u_2(t, U_2)|  = |u_1(t, U_1) - u_1(t, U_2)|  \geq \delta |U_1- U_2|. 
\end{align*}  
 If $s_2(t)x > s_1(t)$, then the extension of $u_1^{-1}$ implies $U_1 = 1$, namely, $u_1(t, U_1) = s_1(t)$.  Clearly, we have $U_1 \geq U_2$ and 
 $u_1(t, U_1) < u_2(t, U_2)$. Hence,  we see that
 \begin{align*}
 0 > & \  u_1(t, U_1) - u_1(t, U_2) + u_1(t, U_2) - u_2(t, U_2) \\
   \geq  & \ \delta (U_1 - U_2) + u_1(t, U_2) - u_2(t, U_2)
 \end{align*}
 so that 
 $$  \delta|U_1 - U_2| \leq |u_1(t, U_2) - u_2(t, U_2)|. $$
 Thus, we get \eqref{40-2}.   Hence, by \eqref{40-1.5}, \eqref{40-2} and \eqref{4a0} we have 
 \begin{align*}
 I_1(t)  & \leq \frac{4 R^2}{\delta^2}  (|u_x(t)|_H^2 + \int_0^1 |u_1(t, u_2^{-1}(t, s_2(t)x)) - u_2(t, u_2^{-1}(t, s_2(t)x))|^2 dx \\
 & \leq \frac{4 R^2}{\delta^2}  (|u_x(t)|_H^2 + \frac{M}{\delta} |u(t)|_H^2) \quad  \mbox{ for } t \in  [0,T]. 
 \end{align*}
 Also, it is obvious that
 $$ I_2(t) \leq \frac{2M}{\delta} |v(t)|_H^2  \quad  \mbox{ for } t \in  [0,T]. $$
Therefore, we have proved this Lemma.  
\end{proof}

Next, we give a lemma concerned with the estimate  for $\bar{w}$.   
\begin{lem}\label{Uni_P2} For any $\varepsilon > 0$ there exists a positive constant $C_3(\varepsilon)$ such that 
\begin{align}
& \frac{1}{2} \frac{d}{dt} |\bar{w}|_H^2 + \frac{\mu \kappa}{2M^2} |\bar{w}_x|_H^2  \nonumber \\
\leq &  C_3(\varepsilon) (|\bar{w}_{2x}|_{H^1(0,1)}^2 + 1) |\bar{w}|_H^2 +  C_3(\varepsilon) (|u_x|_H^2 + |u_t|_H^2) 
+ \varepsilon |u_{tx}|_H^2 \mbox{ a.e. on } [0,T].   \label{4-2-a} 
\end{align}
\end{lem} 
\begin{proof}
Using \eqref{EQ2'},  we have 
\begin{equation}
\bar{w}_t - (\frac{\kappa}{s_1^2} \beta(\bar{w}_1)_{xx} - \frac{\kappa}{s_2^2} \beta(\bar{w}_2)_{xx} ) 
   = - ( \frac{1}{s_1} (\bar{v}_1 \bar{w}_1)_x - \frac{1}{s_2} (\bar{v}_2 \bar{w}_2)_x) 
    + (\frac{s_1'}{s_1} x \bar{w}_{1x} -  \frac{s_2'}{s_2} x \bar{w}_{2x}) \mbox{ a.e. in } Q(T). \label{4-2-b} 
\end{equation}
We multiply \eqref{4-2-b} by $\bar{w}$ and integrate it over [0,1]. Then, we see that 
\begin{align}
& \frac{1}{2} \frac{d}{dt} |\bar{w}|_H^2 \nonumber \\
= & (\frac{\kappa}{s_1^2} \beta(\bar{w}_1)_{xx} - \frac{\kappa}{s_2^2} \beta(\bar{w}_2)_{xx}, \bar{w})_H 
  - ( \frac{1}{s_1} (\bar{v}_1 \bar{w}_1)_x - \frac{1}{s_2} (\bar{v}_2 \bar{w}_2)_x, \bar{w})_H 
    + (\frac{s_1'}{s_1} x \bar{w}_{1x} -  \frac{s_2'}{s_2} x \bar{w}_{2x}, \bar{w})_H  \nonumber \\
=: &   \sum_{i=1}^3  I_i  \quad \mbox{  a.e. on } [0,T].  \label{4-2-c} 
\end{align}
We observe that 
\begin{align}
I_1(t) = & - (\frac{\kappa}{s_1(t)^2} \beta(\bar{w}_1(t))_{x} - \frac{\kappa}{s_2(t)^2} \beta(\bar{w}_2(t))_{x}, \bar{w}_x(t))_H  \nonumber \\
     & + \left(\frac{s_1'(t)}{s_1(t)} \psi(s_1(t)) - \frac{s_2'(t)}{s_2(t)} \psi(s_2(t)) \right) \bar{w}(t,1) 
    -  h_0(t) \left(\frac{1}{s_1(t)} - \frac{1}{s_2(t)} \right) \bar{w}(t,0)  \nonumber \\
    =  & - (\frac{\kappa}{s_1(t)^2} - \frac{\kappa}{s_2(t)^2} )( \beta(\bar{w}_1(t))_{x}, \bar{w}_x(t))_H     
    - \frac{\kappa}{s_2(t)^2} (\beta(\bar{w}_1(t))_{x} - \beta(\bar{w}_2(t))_{x}, \bar{w}_x(t))_H  \nonumber \\
     & - \left(\frac{s_1'(t)}{s_1(t)} \psi(s_1(t)) - \frac{s_2'(t)}{s_2(t)} \psi(s_2(t)) \right) \bar{w}(t,1) 
    -  h_0(t) \left(\frac{1}{s_1(t)} - \frac{1}{s_2(t)} \right) \bar{w}(t,0)  \nonumber  \\
    =:  &\sum_{j=1}^4 I_{1j}(t)   \quad \mbox{  for a.e. }  t \in [0,T],  \label{4-2-d} 
\end{align}
\begin{align}
|I_{11}(t)| \leq & \kappa \left| \frac{s_2(t)^2 - s_1(t) ^2}{s_1(t)^2 s_2(t)^2} \right|  |\beta(\bar{w}_{1})_x(t)|_H |\bar{w}_x(t)|_H \nonumber \\
\leq & \frac{2 \kappa M R }{\delta^4} |s(t)|  |\bar{w}_x(t)|_H  \quad \mbox{  for a.e. }  t \in [0,T],  \label{4-2-e} 
\end{align}
\begin{align}
I_{12}(t) \leq &  - \frac{\kappa}{s_2(t)^2} (\beta'(\bar{w}_1(t)) \bar{w}_{x}(t),  \bar{w}_x(t))_H  
               - \frac{\kappa}{s_2(t)^2} ( (\beta'(\bar{w}_1(t)) -  \beta'(\bar{w}_2(t))) \bar{w}_{2x}(t), \bar{w}_x(t))_H  \nonumber \\
         \leq &  - \frac{\kappa \mu}{M^2} |\bar{w}_{x}(t)|_H^2  
               + \frac{\kappa C_{\rho} }{\delta^2}  |\bar{w}_{2x}(t) \bar{w}(t)|_H |\bar{w}_x(t)|_H  \nonumber \\               
  \leq &  - \frac{3 \kappa \mu}{4M^2} |\bar{w}_{x}(t)|_H^2  + R_1  |\bar{w}_{2x}(t)|_{L^{\infty}(0,1)}^2 |\bar{w}(t)|_H^2  
 \quad \mbox{  for a.e. }  t \in [0,T],  \label{4-2-f} 
\end{align}
where $R_1$ is a positive constant depending only $\mu$, $\kappa$, $M$, $C_{\rho}$ and $\delta$. Also, according to Lemma \ref{lemma1} 
for any $\varepsilon_1 > 0$ and $\varepsilon_2 > 0$ we  see that
\begin{align}
|I_{13}(t)|  \leq &   \left( \left|\frac{1}{s_1(t)} - \frac{1}{s_2(t)} \right|  |s_1'(t)| |\psi(s_1(t))| + \left|\frac{1}{s_2(t)} s'(t) \psi(s_1(t)) \right|  \right. \nonumber \\
   & \left.  \ \ + |\frac{s_2'(t)}{s_2(t)} | |\psi(s_1(t)) - \psi(s_2(t))| \right) |\bar{w}(t,1)|  
 \nonumber \\
 \leq &   \left( \frac{R C_{\psi} }{\delta^2} |s(t)|  + \frac{C_{\psi}}{\delta} |s'(t)|  
    + \frac{R C_{\psi}}{\delta} |s(t)| \right)   |\bar{w}(t,1)| 
 \nonumber \\
 \leq & 2( \frac{R C_{\psi} }{\delta^2} + \frac{C_{\psi}}{\delta} + \frac{R C_{\psi}}{\delta} ) ( |u_x(t)|_H + |u_{tx}(t)|_H)
  (|\bar{w}_x(t)|_H^{1/2} |\bar{w}(t)|_H^{1/2} + |\bar{w}(t)|_H)  \nonumber \\    
  \leq & \varepsilon_1 |\bar{w}_x(t)|_H^2 +  \varepsilon_2 |u_{tx}(t)|_H^2 + R_2(\varepsilon_1, \varepsilon_2) (|u_x(t)|_H^2 + |\bar{w}(t)|_H^2)
 \quad \mbox{  for a.e. }  t \in [0,T],  \label{4-2-g} 
\end{align}
where $R_2(\varepsilon_1, \varepsilon_2)$ is a positive constant depending on $\varepsilon_1$ and 
$\varepsilon_2$. Moreover,  by applying Lemma \ref{lemma1}, again,  it is clear that 
\begin{align}
I_{14}(t) \leq &   |\frac{s(t)}{s_1(t) s_2(t)}|  |h_0(t)| |\bar{w}(t, 0)| \nonumber \\
  \leq &  2\frac{|h_0|_{L^{\infty}(0,T)}}{\delta^2} |u_{x}(t)|_H   (|\bar{w}_x(t)|_H^{1/2} |\bar{w}(t)|_H^{1/2} + |\bar{w}(t)|_H)   \nonumber \\
\leq & \varepsilon_1  |\bar{w}_x(t)|_H^2 +  R_3(\varepsilon_1)  (|u_x(t)|_H^2 + |\bar{w}(t)|_H^2)
   \quad \mbox{  for a.e. }  t \in [0,T],  \label{4-2-h} 
\end{align}
where $R_3(\varepsilon_1)$ is a positive constant depending on $\varepsilon_1$. 

Next, since $s_i'(t) = \bar{v}_i(t,1)$ and $\bar{v}_i(t,0)$ for $t \in [0,T]$ and $i  = 1, 2$, by applying integration by parts we observe that 
%\end{document}
\begin{align}
& I_2(t) + I_3(t) \nonumber \\
 = & (\frac{1}{s_1(t)} \bar{v}_1(t) \bar{w}_1(t) - \frac{1}{s_2(t)} \bar{v}_2(t) \bar{w}_2(t), \bar{w}_x(t))_H 
  - (\frac{s_1'(t)}{s_1(t)} \bar{w}_1(t) - \frac{s_2'(t)}{s_2(t)} \bar{w}_2(t), (x\bar{w}(t))_x)_H \nonumber \\
 = & (\frac{1}{s_1(t)} -  \frac{1}{s_2(t)}) (\bar{v}_1(t) \bar{w}_1(t), \bar{w}_x(t))_H 
 + \frac{1}{s_2(t)} (\bar{v}(t) \bar{w}_1(t), \bar{w}_x(t))_H 
 + \frac{1}{s_2(t)} (\bar{v}_2(t) \bar{w}(t), \bar{w}_x(t))_H \nonumber \\
 & -  (\frac{1}{s_1(t)} -  \frac{1}{s_2(t)}) s_1'(t)  (\bar{w}_1(t), (x\bar{w})_x(t))_H 
 -   \frac{1}{s_2(t)}  s'(t)  (\bar{w}_1(t), (x\bar{w})_x(t))_H \nonumber \\
 & -   \frac{1}{s_2(t)}  s_2'(t)  (\bar{w}(t), (x\bar{w})_x(t))_H \nonumber \\
=: & \sum_{i=1}^6 \hat{I}_i(t) \quad \mbox{  for a.e. }  t \in [0,T].   \label{4-2-i} 
\end{align}
Thanks to Lemma \ref{bar_v},  \eqref{4a0} and \eqref{4a1}, we have
\begin{align}
|\hat{I}_1(t)|  \leq &  \frac{1}{\delta^2} |s(t)|  |\bar{v}_1(t)|_{L^{\infty}(0,1)}  |\bar{w}_1(t)|_H |\bar{w}_x(t)|_H  \nonumber \\
\leq & \frac{R^2}{\delta^2} |u_x(t)|_H  |\bar{w}_x(t)|_H  \nonumber \\
\leq  & \varepsilon_1 |\bar{w}_x(t)|_H^2 + \frac{R^4}{4 \varepsilon_1 \delta^4} |u_x(t)|_H^2
\quad \mbox{  for a.e. }  t \in [0,T],  
 \label{4-2-j} 
\end{align}
\begin{align}
|\hat{I}_2(t)|  \leq &  \frac{1}{\delta}  |\bar{w}_1(t)|_{L^{\infty}(0,1)}  |\bar{v}(t)|_H |\bar{w}_x(t)|_H  \nonumber \\ 
\leq  & \varepsilon_1 |\bar{w}_x(t)|_H^2 + \frac{R^2C_2^2 }{2 \varepsilon_1 \delta^2} (|u_x(t)|_H^2  + |u_t(t)|_H^2)
\quad \mbox{  for a.e. }  t \in [0,T],  
 \label{4-2-k} 
\end{align}
\begin{align}
|\hat{I}_3(t)|  \leq &  \frac{1}{\delta}  |\bar{v}_2(t)|_{L^{\infty}(0,1)}  |\bar{w}(t)|_H |\bar{w}_x(t)|_H  \nonumber \\ 
\leq  & \varepsilon_1 |\bar{w}_x(t)|_H^2 + \frac{R^2 }{4 \varepsilon_1 \delta^2} |\bar{w}(t)|_H^2  
\quad \mbox{  for a.e. }  t \in [0,T],  
 \label{4-2-l} 
\end{align}
\begin{align}
|\hat{I}_4(t)|  \leq &  \frac{1}{\delta^2} |s(t)|  |s_1'(t)|  |\bar{w}_1(t)|_{H}  (|\bar{w}(t)|_H + |\bar{w}_x(t)|_H)  \nonumber \\ 
\leq  & \varepsilon_1 |\bar{w}_x(t)|_H^2 + (\frac{R^4 }{4 \varepsilon_1 \delta^4} +   \frac{R^4 }{2 \delta^4}) 
                 |u_x(t)|_H^2 + \frac{1}{2} |\bar{w}(t)|_H^2  
\quad \mbox{  for a.e. }  t \in [0,T],  
 \label{4-2-m} 
\end{align}
\begin{align}
|\hat{I}_5(t)|  \leq &  \frac{1}{\delta} |s'(t)|   |\bar{w}_1(t)|_{H}  (|\bar{w}(t)|_H + |\bar{w}_x(t)|_H)  \nonumber \\ 
\leq  &  \frac{2R}{\delta} |u_t(t)|_H^{1/2} |u_{tx}(t)|_H^{1/2}    (|\bar{w}(t)|_H + |\bar{w}_x(t)|_H)  \nonumber \\
\leq  &   \varepsilon_1 |\bar{w}_x(t)|_H^2 +  \varepsilon_2 |u_{tx}(t)|_H^2 + R_4(\varepsilon_1, \varepsilon_2)  
( |u_t(t)|_H^2 +  |\bar{w}(t)|_H^2)  \quad \mbox{  for a.e. }  t \in [0,T],  
 \label{4-2-n} 
\end{align}
\begin{align}
|\hat{I}_6(t)|  \leq &  \frac{1}{\delta} |s_2'(t)|   |\bar{w}(t)|_{H}  (|\bar{w}(t)|_H + |\bar{w}_x(t)|_H)  \nonumber \\ 
\leq  &   \varepsilon_1 |\bar{w}_x(t)|_H^2 + (\frac{R^2}{4 \varepsilon_1 \delta^2} + \frac{R}{\delta}) |\bar{w}(t)|_H^2  \quad \mbox{  for a.e. }  t \in [0,T],  
 \label{4-2-o} 
\end{align}
where $R_4(\varepsilon_1, \varepsilon_2)$ is a positive constant depending on $\varepsilon_1$  and  $\varepsilon_2$.

From \eqref{4-2-c} $\sim$ \eqref{4-2-o} it follows that 
\begin{align*}
& \frac{1}{2} \frac{d}{dt} |\bar{w}|_H^2 + \frac{3 \kappa \mu}{4M^2} |\bar{w}_{x}|_H^2 \\
\leq &  8  \varepsilon_1 |\bar{w}_x|_H^2  + 2 \varepsilon_2 |u_{tx}|_H^2 
 + R_5(\varepsilon_1, \varepsilon_2) ( |\bar{w}_{2x}|_{H^1(0,1)}^2 + 1) |\bar{w}|_H^2   \\
 & \ +  R_5(\varepsilon_1, \varepsilon_2)(|u_x|_H^2 + |u_t(t)|_H^2) 
 \quad \mbox{   a.e.  on }  [0,T],  
\end{align*}
where $R_5(\varepsilon_1, \varepsilon_2)$ is a positive constant depending on $\varepsilon_1$  and  $\varepsilon_2$.   Here, by taking 
$\varepsilon_1 > 0$ with $\displaystyle 9  \varepsilon_1 \leq  \frac{ \kappa \mu}{4M^2}$,  we can obtain \eqref{4-2-a}.  
Thus, we have proved Lemma \ref{Uni_P2}. 
\end{proof} 

\vskip 12pt
\begin{proof}[Proof of Theorem \ref{main_th}(uniqueness)] 
Lemmas \ref{Uni_P1} and \ref{Uni_P2} imply that 
\begin{align}
& \frac{d}{dt} ( \frac{m}{2} |u_t|_H^2 + \frac{\gamma}{2} |u_{xx}|_H^2 + \frac{k}{4} |u_x|_H^2 + \frac{1}{2} |\bar{w}|_H^2) 
   + \frac{k_v}{4} |u_{tx}|_H^2 + \frac{\kappa \mu}{2 M^2} |\bar{w}_x|_H^2  \nonumber \\
 \leq & C_4(1 + |\bar{p}_{1x}|_{L^{\infty}(0,1)}^2 ) |u_x|_H^2  +  C_4  |u_t|_H^2
 + C_4 (|\bar{w}_{2x}|_{H^1(0,1)}^2 + 1) |\bar{w}|_H^2  \quad \mbox{   a.e.  on }  [0,T],  \nonumber 
\end{align}
where $C_4$ is a positive constant. 
The uniqueness of the solutions is a direct application of  Gronwall's inequality.  
\end{proof}

\section{ Solvability of (P1)}  \label{P1} 
In our proof of Theorem \ref{main_th}, making use of weak solutions plays a very important role. Here, we give a definition of weak solutions to (P1).
\begin{dfn} \label{def_weak_P1} \rm
Let $T > 0$ and $\hat{p}$, $u_0$ and $v_0$ be given functions. 
We call a function $u$ on $Q(T)$  a weak solution of (P1)($\nu(\hat{p}), u_0,  v_0$), if the following conditions hold: 
\begin{itemize}
\item[(W1)]$u \in W^{1,\infty}(0, T; H) \cap L^{\infty}(0, T; X) \cap W^{1,2}(0, T; H^1(0,1)) =: W_1(T)$.
\item[(W2)] $u_x > 0$ on $\overline{Q(T)}$. 
\end{itemize}
$$ \leqno{\mbox{(W3)}}  
\left\{ 
\begin{array}{ll}
& \displaystyle  m \int_0^T (u_t, \eta_t)_H dt + \gamma \int_0^T (u_{xx}, \eta_{xx})_H dt + \int_0^T (f(u_{x}) + k_v u_{tx}, \eta_x)_H dt \\[0.2cm]
= & \displaystyle  - \int_0^T \varphi(u(t,1)) \eta(t,1) dt - \int_0^T ( \nu(\hat{p}), \eta_x)_H dt + m(v_0, \eta(0))_H \\[0.2cm]
& \qquad \mbox{ for } \eta \in  W^{1,2}(0, T; H) \cap L^{\infty}(0,T; X) \mbox{ with } \eta(T) = 0.
\end{array} \right. 
$$

\end{dfn}

%%Here, we note that $u_x$ is continuous on $\overline{Q(T)}$ for $u \in W_1(T)$. 

The aim of this section is to prove Proposition \ref{uni_weak_P1} and Proposition \ref{exist_weak_P1}. 
First, we show the uniqueness of weak solutions to (P1) by applying the dual equation method - a trick needed to compensate for the  lack of regularity of weak solutions.% (standard methods for uniqueness does work well so that we prove this lemma. 
  
\begin{prop} \label{uni_weak_P1} 
Assume (A2) and (A3). 
If $u_0 \in X$ with $u_{0x} > 0$ on $[0,1]$, $v_0 \in H$ and $\hat{p} \in L^2(0,T; H^1(0,1))$, then 
(P1)($\hat{p}, u_0,  v_0$)  admits at most one weak solution. 
\end{prop}

\begin{proof}
Let $u_1$ and $u_2$ be weak solutions of (P1)($\hat{p}, u_0,  v_0$) on $[0,T]$ and put $u = u_1 - u_2$.  
By (W1) and (W2) we can take a positive constant $\delta$ satisfying $u_{ix} \geq \delta$ on $\overline{Q(T)}$ for $i = 1, 2$. 
Also, let $\eta \in  W^{2,2}(0, T; H) \cap L^{\infty}(0,T; X \cap H^4(0,1)) \cap W^{1,2}(0, T; H^2(0,1))$ with $\eta(T) = \eta_t(T) = 0$ and 
 $\eta_{xx}(\cdot, 0) = \eta_{xx}(\cdot,1) = 0$ on $[0,T]$. 
From (W3) it is clear that 
 \begin{align}
 & - m \int_0^T (u_t(t), \eta_t(t))_H dt + \gamma \int_0^T (u_{xx}(t), \eta_{xx}(t))_H dt 
  +  \frac{k}{2} \int_0^T (u_x(t), \eta_x(t))_H dt  \nonumber \\
&   - \frac{k}{4} \int_0^T ( \frac{1}{u_{1x}(t)^3} - \frac{1}{u_{2x}(t)^3}, \eta_x(t) )_H dt  
+ k_v \int_0^T (u_{tx}(t), \eta_{x}(t))_H dt \  (=:  \sum_{i=1}^5 I_i)    \nonumber \\
= & - \int_0^T ( \varphi(u_1(t,1)) - \varphi(u_2(t,1)) ) \eta(t,1) dt  \ (=: I_6).   \nonumber
 \end{align}
By elementary calculations,  we obtain that 
\begin{align*}
 I_1 + I_2 
=  m \int_0^T (u, \eta_{tt})_H dt + \gamma \int_0^T (u, \eta_{xxxx})_H dt - \gamma \int_0^T u(t,1) \eta_{xxx}(t,1) dt,  
\end{align*}
\begin{align*}
 I_3 
=  -\frac{k}{2}  \int_0^T (u(t), \eta_{xx}(t))_H dt + \frac{k}{2} \int_0^T u(t,1) \eta_{x}(t,1) dt,  
\end{align*}
\begin{align}
 I_4 
& =  \frac{k}{4}  \int_0^T (u_x(t),  F(t) \eta_{x}(t))_H dt  \nonumber \\
& = -  \frac{k}{4}  \int_0^T (u(t),  (F(t) \eta(t))_{x})_H dt + \frac{k}{4} \int_0^T u(t,1) F(t,1)  \eta_{x}(t,1) dt,  \label{5q-2}
\end{align}
\begin{align*}
 I_5 
=  k_v  \int_0^T (u(t), \eta_{txx}(t))_H dt -  k_v \int_0^T u(t,1) \eta_{tx}(t,1) dt,  
\end{align*}
where $\displaystyle F = \frac{u_{1x}^2 + u_{1x} u_{2x} + u_{2x}^2}{u_{1x}^3 u_{2x}^3}$ on $Q(T)$.  Since 
$F \in C(\overline{Q(T)}) \cap L^{\infty}(0,T; H^1(0,1))$, the first term in the right hand side of \eqref{5q-2}  is well-defined. 
Here, we put 
$$ b = \left\{ \begin{array}{cl} 
     \displaystyle -\frac{\varphi(u_1(t,1)) - \varphi(u_2(t,1))}{u(t,1)} & \mbox{ if } u(t,1) \ne 0, \\
     0 & \mbox{ otherwise}, 
     \end{array} \right. $$
and then we have $b \in L^{\infty}(0,T)$. Accordingly, we see that 
$$ I_6 = \int_0^T b(t) u(t,1) \eta(t,1) dt. $$
Hence, from these equations it holds that 
\begin{align}
& \int_0^T (u, m\eta_{tt} + \gamma \eta_{xxxx} - \frac{k}{2} \eta_{xx} - \frac{k}{4} (F\eta_x)_x + k_v \eta_{xx})_H dt \nonumber \\
= & \int_0^T u(\cdot,1) \left( b \eta(\cdot,1) + \gamma \eta_{xxx}(\cdot,1) - \frac{k}{2} \eta_x(\cdot, 1) - \frac{k}{4} F(\cdot,1) \eta_x(\cdot,1)
                  + k_v \eta_{tx}(\cdot, 1)  \right) dt \label{5q-3} 
\end{align}
for  $\eta \in  W^{2,2}(0, T; H) \cap L^{\infty}(0,T; X \cap H^4(0,1)) \cap W^{1,2}(0, T; H^2(0,1))$ with $\eta(T) = \eta_t(T) = 0$ and 
 $\eta_{xx}(\cdot, 0) = \eta_{xx}(\cdot,1) = 0$ on $[0,T]$.

Here, we can take $\{b_n\} \subset C^{\infty}([0,T])$ and $\{F_n\} \subset C^{\infty}(\overline{Q(T)})$ such that
$\{b_n\}$ and $\{F_n\}$ are bounded in $L^{\infty}(0,T)$ and $L^{\infty}(Q(T))$, respectively, 
$b_n \to b$ in $L^2(0,T)$  and $F_n \to F$ in $L^2(0,T; H)$  as $n \to \infty$.  Let $\zeta \in C_0^{\infty}(Q(T))$ and 
(DP)$_n(\zeta)$ be the following problem for each $n$: 
 \begin{align}
 & m \eta_{tt}^{(n)} + \gamma \eta_{xxxx}^{(n)}  - \frac{k}{2} \eta_{xx}^{(n)} - \frac{k}{4} (F_n \eta_x^{(n)})_x + k_v \eta_{txx}^{(n)} = \zeta 
     \mbox{ in } Q(T),  \label{5q-4} \\
 & \eta_{xx}^{(n)}(\cdot, 0) = \eta_{xx}^{(n)}(\cdot,1) = 0, \eta^{(n)}(\cdot,0) = 0 \mbox{ on  } [0,  T],  \nonumber  \\
 & b_n \eta^{(n)}(\cdot,1) + \gamma \eta_{xxx}^{(n)}(\cdot,1) - \frac{k}{2} \eta_x^{(n)}(\cdot, 1) - \frac{k}{4} F_n(\cdot,1) \eta_x^{(n)}(\cdot,1)
                  + k_v \eta_{tx}^{(n)}(\cdot, 1) = 0 \mbox{ on }  [0,  T],  \label{5q-44}\\
  & \eta^{(n)}(T) = \eta_t^{(n)}(T) = 0. \nonumber 
 \end{align}
 Since (DP)$_n$ is linear and has smooth coefficients,  (DP)$_n$ has a unique strong solution $\eta^{(n)}$ for each $n$.  
 Moreover, we see that 
 $\{\eta^{(n)} \}$ is bounded in $L^{\infty}(0,T; X)$,  and 
 namely, $|\eta^{(n)}(t)|_X \leq C'$ for $t \in  [0, T]$ and any $n$, where $C'$ is some positive constant.  
 A proof of this boundedness is given as Lemma \ref{lem_unif_DP}.  Hence, \eqref{5q-3} and \eqref{5q-4} imply that 
\begin{align*}
& \int_{Q(T)} u \zeta dx dt \\
= & \int_{Q(T)} u ( m \eta_{tt}^{(n)} + \gamma \eta_{xxxx}^{(n)}  - \frac{k}{2} \eta_{xx}^{(n)} - \frac{k}{4} (F_n \eta_x^{(n)})_x + k_v \eta_{txx}^{(n)} )dx  dt\\
= & -\frac{k}{4} \int_{Q(T)} u ( (F_n \eta_x^{(n)} )_x - (F \eta_x^{(n)})_x) dx dt
  + \int_0^T u(\cdot, 1) \eta^{(n)}(\cdot,1) ( b - b_n)  dt \\
& \  -\frac{k}{4}  \int_0^T u(\cdot, 1) \eta_x^{(n)}(\cdot,1) ( F(\cdot, 1) - F_n(\cdot,1))  dt \\
= & \frac{k}{4} \int_{Q(T)} u_x (F_n  - F) \eta_x^{(n)}dx dt
  + \int_0^T u(\cdot, 1) \eta^{(n)}(\cdot,1) ( b - b_n)  dt   \quad \mbox{ for } n. 
\end{align*}
By the definitions of $F_n$ and $b_n$ and uniform estimate for $\eta^{(n)}$ we have
$$ \int_{Q(T)} u \zeta \ dx dt  = 0  \mbox{ for any } \zeta \in C_0^{\infty}(Q(T)) $$
so that $u  = 0$ on $Q(T)$.  This is a conclusion of the present proposition.  
\end{proof} 

{
The method applying the dual equation is found in the book by Ladyzenskaja, Solonnikov and Ural’ceva \cite{LSU}. Also, by applying this method,  
 Niezgodka and Pawlow \cite{NP} obtained uniqueness of solutions of the Stefan problem, and Aiki \cite{Ai(Shape)}  proved uniqueness in the system for shape memory alloy materials.  
}

\begin{lem} \label{lem_unif_DP} 
There exists a positive constant $C'$ independent of $n$ such that 
$$ |\eta^{(n)}(t)|_X \leq C' \mbox{  for } t \in [0,T] \mbox{ and any  } n. $$ 
\end{lem}
\begin{proof}
First, thanks to the definitions of $F_n$ and $b_n$ we can choose a positive constant $\hat{C}$ satisfying
$$ 
|F_n(t)|_{L^{\infty}(0,1)} \leq \hat{C} \mbox{ and } |b_n(t)| \leq \hat{C} \mbox{ for } t \in [0,T]  \mbox{ and } n = 1, 2, \cdots.
$$

By putting $\hat{\eta}^{(n)}(t, x) = \eta^{(n)}(T-t, x)$, $\hat{F}_n(t, x) = F_n(T-t, x)$, $\hat{\zeta}(t, x) = \zeta(T-t, x)$  
for $(t, x) \in Q(T)$ and $\hat{b}_n(t) = b_n(T-t)$ for $t \in [0,T]$, it holds that 
  \begin{align}
 & m \hat{\eta}_{tt}^{(n)} + \gamma \hat{\eta}_{xxxx}^{(n)}  - \frac{k}{2} \hat{\eta}_{xx}^{(n)} - \frac{k}{4} (\hat{F}_n \hat{\eta}_x^{(n)})_x 
         - k_v \hat{\eta}_{txx}^{(n)} = \hat{\zeta} 
   \mbox{ in } Q(T),  \label{5r-4} \\
 & \hat{\eta}_{xx}^{(n)}(t,0) = \hat{\eta}_{xx}^{(n)}(t,1) = 0, \hat{\eta}^{(n)}(t,0) = 0 \mbox{ for } 0 \leq t \leq T,  \nonumber  \\
 & \hat{b}_n \hat{\eta}^{(n)}(\cdot,1) + \gamma \hat{\eta}_{xxx}^{(n)}(\cdot,1) - \frac{k}{2} \hat{\eta}_x^{(n)}(\cdot, 1) 
      - \frac{k}{4} \hat{F}_n(\cdot,1) \hat{\eta}_x^{(n)}(\cdot,1)
                   - k_v \hat{\eta}_{tx}^{(n)}(\cdot, 1) = 0 \mbox{ for } 0 \leq t \leq T,  \nonumber \\
  & \hat{\eta}^{(n)}(0) = \hat{\eta}_t^{(n)}(0) = 0. \nonumber 
 \end{align}
We multiply \eqref{5r-4} by $\hat{\eta}_t^{(n)}$ and integrate it over $(0,1)$. By integrating by parts and boundary conditions, we see that 
\begin{align*}
& \frac{d}{dt} ( \frac{m}{2} |\hat{\eta}_t^{(n)}|_H^2 + \frac{\gamma}{2}  |\hat{\eta}_{xx}^{(n)}|_H^2 + \frac{k}{4} |\hat{\eta}_{x}^{(n)}|_H^2)
   + k_v |\hat{\eta}_{tx}^{(n)}|_H^2 \\
= & -\frac{k}{4} (\hat{F}_n    \hat{\eta}_x^{(n)},  \hat{\eta}_{tx}^{(n)})_H + \hat{b}_n \hat{\eta}_t^{(n)}(\cdot,1) + 
        (\hat{\zeta},  \hat{\eta}_{t}^{(n)})_H \\
\leq &  \frac{k_v}{2} |\hat{\eta}_{tx}^{(n)}|_H^2 + \frac{k^2}{32 k_v} |\hat{F}_n|_{L^{\infty}(0,1)}^2 |\hat{\eta}_{x}^{(n)}|_H^2 
    + \frac{1}{k_v} |\hat{b}_n|^2 + \frac{k_v}{4} |\hat{\eta}_{tx}^{(n)}|_H^2 
    + \frac{1}{2} |\hat{\zeta}|_H^2 + \frac{1}{2} |\hat{\eta}_{t}^{(n)}|_H^2 
\end{align*}
 and 
 \begin{align*}
& \frac{d}{dt} ( \frac{m}{2} |\hat{\eta}_t^{(n)}|_H^2 + \frac{\gamma}{2}  |\hat{\eta}_{xx}^{(n)}|_H^2 + \frac{k}{4} |\hat{\eta}_{x}^{(n)}|_H^2)
   + \frac{k_v}{4} |\hat{\eta}_{tx}^{(n)}|_H^2 \\
\leq &  C''( |\hat{\eta}_{x}^{(n)}|_H^2 + |\hat{\eta}_{t}^{(n)}|_H^2 + 1)
    + \frac{1}{2} |\hat{\zeta}|_H^2  \quad \mbox{ on } [0,T], 
\end{align*}
where $C''$ is a positive constant depending on $\hat{C}$. Therefore, this Lemma is a direct application of Gronwall's inequality.
\end{proof}

\begin{prop} \label{exist_weak_P1} 
Assume (A2) and (A3). 
If $u_0 \in X$ with $u_{0x} > 0$ on $[0,1]$, $v_0 \in H$ and $\hat{p} \in L^2(0,T; H^1(0,1))$, then 
(P1)($\hat{p}, u_0,  v_0$)  has at least one weak solution. 
\end{prop}

We prove this Proposition by applying the Galerkin method. 
Since $X$ is a separable Hilbert space, we can take a complete orthonormal system $\{\hat{z}_i \}$ and denote its normalization in $H$ by $\{z_i \}$. Also, let
$X_n$ be a subspace generated by $\{z_i | i= 1, 2, \cdots, n \}$. 
First, we assume that $u_0 \in X$ with $u_{0x} > 0$ on $[0,1]$, $v_0 \in X$ and $\hat{p} \in L^2(0, T; X)$. It is easy to see that 
\begin{equation}  \left. 
\begin{array}{l} 
u_{0n} := \sum_{i=1}^n (u_0, \hat{z}_i)_X \hat{z}_i \to u_0, v_{0n} := \sum_{i=1}^n (v_0, \hat{z}_i)_X \hat{z}_i \to v_0 \mbox{ in } X, \\[0.2cm]
\hat{p}_{n}:= \sum_{i=1}^n( \hat{p}, \hat{z}_i)_X  \hat{z}_i  \to \hat{p} \mbox{ in } L^2(0, T; X) 
\end{array} \right\} 
 \mbox{ as } n \to  \infty,  \label{def_ap}
 \end{equation}
and $u_{0n}, v_{0n} \in X_n$ and $\hat{p}_n \in X_n$ on $[0,T]$ for each $n$. 
 Since $C^1([0,1]) \subset X$,  we have $u_{0nx} \to u_{0x}$ in $C([0,1])$ as $n \to \infty$ so that there exist $\delta > 0$ and $n_0$ such that 
 $u_{0nx} \geq \delta$ on $[0,1]$ for $n \geq n_0$. Moreover, let $g \in  C([0,T])$ and put 
 $$ u_{0n} = \sum_{i=1}^n a_{0i}^{(n)} z_i,  v_{0n} = \sum_{i=1}^n \alpha_{0i}^{(n)} z_i \mbox{ for } n \geq n_0, $$
 where $a_i^{(n)}$ and  $\alpha_i^{(n)}$ are constants. 
 
 By using these notation we introduce the following problem for $n \geq n_0$. 
 The  problem (AP1)$_n :=$ (AP1)$_n(u_{0n}, v_{0n}, \hat{p}_n, g)$ is to find a function $u_n = \sum_{i=1}^n a_i^{(n)} z_i$, 
 $a_i ^{(n)} \in C^2([0,T])$,  such that 
 \begin{align}
 & \ m(u_{ntt}(t), z_i)_H + \gamma (u_{nxx}(t), z_{ixx})_H +  ( f(u_{nx}(t)), z_{ix})_H + k_v (u_{ntx}(t), z_{ix})_H + (\hat{p}_n(t), z_{ix})_H \nonumber \\
 = & \  g(t)  z_i(1) \mbox{ for  }  t \in [0,T] \mbox{ for } i = 1,  2, \cdots, n, \label{ODE1}  \\  
  & u_{n}(0) = u_{0n}, u_{nt}(0) = v_{0n}.  \label{ODE2} 
 \end{align}
 
 The next Lemma guarantees existence of solutions of (AP1)$_n$. 
\begin{lem} \label{exist_ODE} 
For $n \geq n_0$ (AP1)$_n(u_{0n}, v_{0n}, \hat{p}_n, g)$ has a solution 
$u_n \in C^2([0,T]; X)$ satisfying $u_n(t) \in X_n$ for $t \in [0,T]$. 
\end{lem}  
\begin{proof}
By substituting $u_n = \sum_{i=1}^n a_i^{(n)} z_i$ into \eqref{ODE1} we obtain the following ordinary differential equations for 
$a^{(n)} = (a_1^{(n)}, \cdots, a_n^{(n)})$: 
\begin{equation}
\left\{
\begin{array}{l}
\displaystyle \frac{d^2}{dt^2} a^{(n)} = - A a^{(n)} -  F(a^{(n)}, \frac{d}{dt} a^{(n)} ) + G \quad \mbox{ on } [0, T] , \\
\displaystyle  a^{(n)}(0) = (a_{01}^{(n}, \cdots, a_{0n}^{(n)}), \frac{d}{dt} a^{(n)}(0) = (\alpha_{01}^{(n}, \cdots, \alpha_{0n}^{(n)}), 
\end{array}
\right. \nonumber
\end{equation}
where $A$ is a ($n \times n$)-matrix with components $(z_{ixx}, z_{jxx})_H$, $F: {\mathbb R}^{2n} \to {\mathbb R}^n$ and 
$G \in C([0,T]; {\mathbb R}^n)$ are given by 
$\displaystyle F( a^{(n)}, \frac{d}{dt} a^{(n)}) =  ( (f( \sum_{i=1}^n a_i^{(n)} z_{ix}) + k_v \sum_{i=1}^{(n)} \frac{d}{dt} a_i^{(n)} z_{ix} , z_{jx})_H)_{j}$, 
and  \\ $G(t) = (g(t) z_j(1) - (\hat{p}_n(t), z_j)_H)_{j} $ for $t \in [0,T]$.  Since $F$ is locally Lipschitz continuous and $u_{0nx} \geq \delta$ on $[0,1]$, it is clear that a unique solution $u_n$ of (AP1)$_n$ with $u_{nx} > 0$ on $\overline{Q(T')}$ exists on $[0,T']$ for some 
$0 < T' \leq T$. In order to prove existence of the solution on the whole interval $[0,T]$, it is sufficient to show that there exist $\delta' > 0$ and $M > 0$ which are independent of $T'$ such that 
\begin{equation}
\delta' \leq u_{nx} \leq M  \mbox{ on } Q(T'). \label{lem5-1-a}
\end{equation}
Since we can substitute $u_{nt}$ instead of $z_i$ in \eqref{ODE1}, it yields that 
$$
\begin{array}{ll}
& \displaystyle 
 \frac{m}{2} \frac{d}{dt} |u_{nt}(t)|_H^2 + \frac{\gamma}{2}  \frac{d}{dt}  |u_{nxx}(t)|_H^2 + (f(u_{nx}(t)) + k_v u_{ntx}(t), u_{tx}(t))_H  + (\hat{p}_n(t), u_{ntx}(t))_H 
 \\[0.2cm]
= & g(t)u_{nt}(t,1)  \quad \mbox{ for } t \in [0, T']. 
\end{array} 
$$
Here, we note that
$$ 
(f(u_{nx}(t)), u_{tx}(t))_H = \frac{k}{4}  \frac{d}{dt} ( |u_{nx}(t)|_H^2 - |u_{nx}(t)|_{L^1(0,1)} + \frac{1}{2} \int_0^1 \frac{1}{|u_{nx}(t)|^2} dx) 
          \mbox{ for } t \in [0, T'].  
$$ 
Accordingly, we see that 
\begin{align}
& \frac{d}{dt} \left( \frac{m}{2}  |u_{nt}|_H^2 + \frac{\gamma}{2}   |u_{nxx}|_H^2 +
 \frac{k}{4} (|u_{nx}|_H^2 - |u_{nx}|_{L^1(0,1)} + \frac{1}{2} \int_0^1 \frac{1}{|u_{nx}|^2} dx) \right) +  k_v |u_{ntx}|_H^2  \nonumber  \\
 \leq & |\hat{p}_n|_H |u_{ntx}|_H + |g| |u_{ntx}|_H  \quad   \mbox{ on }  [0, T'],  \nonumber
\end{align}
and 
\begin{align}
&  \frac{m}{2}  |u_{nt}(t)|_H^2 + \frac{\gamma}{2}   |u_{nxx}(t)|_H^2 +
 \frac{k}{4} (|u_{nx}(t)|_H^2 + \frac{1}{2} \int_0^1 \frac{1}{|u_{nx}(t)|^2} dx ) +  \frac{k_v}{2} \int_0^t |u_{ntx}|_H^2 d\tau \nonumber  \\
 \leq & \frac{m}{2}  |v_{0n}|_H^2 + \frac{\gamma}{2}   |u_{0nxx}|_H^2 + 
 \frac{k}{4} (|u_{0nx}|_H^2  + \frac{1}{2} \int_0^1 \frac{1}{|u_{0nx}|^2} dx  +    |u_{nx}(t)|_{L^1(0,1)})  \label{5-1-x}\\
 & + \frac{1}{k_v} \int_0^t  (|\hat{p}_n|_H^2  +  |g|^2) d\tau 
 \quad   \mbox{ for } t \in [0,  T'].  \nonumber 
\end{align}
Hence, there exists a positive constant $C$ independent of $T'$ such that 
$$  |u_{nt}(t)|_H^2 +   |u_{nxx}(t)|_H^2 + |u_{nx}(t)|_H^2 + \int_0^1 \frac{1}{|u_{nx}(t)|^2} dx  \leq 
C  \quad   \mbox{ for } t \in [0,  T']. $$
By applying Lemma \ref{key_lemma}, we can get \eqref{lem5-1-a} for some $\delta' > 0$ and $M > 0$ which are independent of $T'$. 
Thus, we can show existence of $u_n \in C^2([0,T]; X)$ satisfying $u_n(t) \in X_n$ for $t \in [0,T]$,  \eqref{ODE1} and \eqref{ODE2}. 

\end{proof}

To move on with the proof of Proposition \ref{exist_weak_P1}, we need the estimate for the difference of {the} solutions. 
\begin{lem} For 
$i = 1, 2$  let $g_i \in C([0,T])$, $u_{in}$ be a solution of (AP1)$_n(u_{0n}, v_{0n}, \hat{p}_n, g_i)$ for $i = 1, 2$ and $n \geq n_0$, 
and $u_n = u_{1n} - u_{2n}$. 
There exist a positive constant $C_*$ depending on the lower bound of $u_{inx}$ and $|u_{in}|_{L^{\infty}(0,T; X)}$ for $i = 1, 2$ such that 
\begin{align}
 |u_{nt}(t)|_H^2 + |u_{nxx}(t)|_H^2 + |u_{nx}(t)|_H^2 + \int_0^t |u_{ntx}(\tau)|_H^2 d\tau 
\leq C_* \int_0^t |g_1 - g_2|^2 d\tau \quad \mbox{ for } t \in [0, T].  \label{diff_P1_lem} 
\end{align} 
\end{lem}
\begin{proof}
Let $u_{jn}$ be a solution of (AP1)$_n(u_{0n}, v_{0n}, \hat{p}_n, g_j)$ for $j = 1, 2$ and $n \geq n_0$ and  put $u_n = u_{1n} - u_{2n}$ and 
$g = g_1 - g_2$.
Since $u_{nt}(t) \in X_n$ for $t \in [0,T]$, by \eqref{ODE1} we see that 
 \begin{align}
 & \ m(u_{ntt}, u_{nt})_H + \gamma (u_{nxx}, u_{ntxx})_H +  ( f(u_{1nx})  - f(u_{2nx}), u_{nxt})_H 
        + k_v |u_{ntx}|_H^2  \nonumber \\
 = & \  g u_{nt}(\cdot,1)  \nonumber \\
 \leq & \frac{k_v}{4} |u_{ntx}|_H^2  + \frac{1}{ k_v} |g|^2  \mbox{ on  }  [0,T]. \label{5z-1}  
 \end{align}
 Here, we note that 
 \begin{align*}
( f(u_{1nx})  - f(u_{2nx}), u_{nxt})_H  
 = & \frac{k}{4}  \frac{d}{dt} |u_x|_H^2 + \frac{k}{4} (\frac{u_{1nx}^2 - u_{2nx}^2}{ |u_{1nx}u_{2nx}|^3}, u_{ntx})_H \\
 \geq &  \frac{k}{4} \frac{d}{dt}  |u_x|_H^2 - \frac{k_v}{4} |u_{ntx}|_H^2 -  \frac{k^2M^4}{k_v|\delta'|^{12}} |u_{nx}|_H^2,  
 \end{align*}
where $\delta'$ and $M$ are positive constants used in the proof of Lemme \ref{exist_ODE}   and then we have 
 \begin{align}
 & \  \frac{d}{dt} \left( \frac{m}{2} |u_{nt}|_H^2  + \frac{\gamma}{2}  |u_{nxx}|_H^2 + \frac{k}{4} |u_x|_H^2 \right) 
 +  \frac{k_v}{2} |u_{ntx}|_H^2  \nonumber \\
 \leq &\  \frac{k^2M^4}{k_v  |\delta'|^{12}} |u_x|_H^2 + \frac{1}{ k_v} |g|^2 \mbox{ on }   [0,T]. \nonumber
\end{align}
Hence, Gronwall's inequality implies 
 \begin{align}
  \frac{m}{2} |u_{nt}(t)|_H^2  + \frac{\gamma}{2}  |u_{nxx}(t)|_H^2 + \frac{k}{4}  |u_x(t)|_H^2
 +  \frac{k_v}{2} \int_0^t |u_{ntx}|_H^2 d\tau  
 \leq  \frac{e^{C_1 t}}{ k_v} \int_0^t |g|^2  d\tau \mbox{ for } t \in  [0,T],  \nonumber 
\end{align}
where $\displaystyle C_1= \frac{4kM^4}{k_v  |\delta'|^{12}}$. Thus, we have proved this Lemma. 
\end{proof}

Next, we consider the solvability of the following auxiliary problem (AP1) $:=$ (AP1)$(u_0, v_0, g, q)$ in  a weak sense: 
\begin{align}
& m u_{tt}  + \gamma u_{xxxx} -  ( f(u_{x}) + k_v u_{xt})_x =  g_{x}   \mbox{ in } Q(T),  \label{5m-1} \\
& u(t, 0) = 0, u_{xx}(t,0) = u_{xx}(t,1) = 0 \mbox{ for } t \in [0,T],   \label{5m-2} \\
&  - \gamma u_{xxx}(t,1) +  f(u_{x})(t,1) + k_v u_{xt}(t,1) + g(t, 1)  = q(t) \mbox{ for } t \in [0,T],    \label{5m-3}\\
&  u(0) = u_{0}, u_t(0) = v_{0}.  \nonumber
\end{align}

Here, we give a definition of a weak solution to (AP1). 
\begin{dfn} 
We call that  $u$ is  a weak solution of (AP1)$(u_0, v_0, g, q)$ on $[0,T]$, if $u$ satisfies (W1), (W2) and 
\begin{align}
& \displaystyle  m \int_0^T (u_t, \eta_t)_H dt + \gamma \int_0^T (u_{xx}, \eta_{xx})_H dt + \int_0^T (f(u_{x}) + k_v u_{tx}, \eta_x)_H dt 
    \nonumber \\[0.2cm]
= & \displaystyle   \int_0^T q(t) \eta(t,1) dt - \int_0^T (g, \eta_x)_H dt + m(v_0, \eta(0))_H  \label{def_AP1}\\[0.2cm]
 & \qquad \mbox{ for } \eta \in  W^{1,2}(0, T; H) \cap L^{\infty}(0,T; X) \mbox{ with } \eta(T) = 0. \nonumber
\end{align}
 
\end{dfn}

\begin{lem} \label{AP1}
If $\hat{p} \in L^2(0, T; H)$, $g \in L^2(0,T)$, $u_0 \in X$ with $u_{0x} > 0$ on $[0,1]$ and $v_0 \in H$, then there exists a unique weak solution $u$ of (AP1)$(u_0, v_0, \hat{p}, g)$ on  $[0,T]$, and the following inequality holds: 
\begin{align}
&  \frac{m}{2}  |u_{t}(t)|_H^2 + \frac{\gamma}{2}   |u_{xx}(t)|_H^2 +
 \frac{k}{8} (|u_{x}(t)|_H^2 +  \int_0^1 \frac{1}{|u_{x}(t)|^2} dx ) +  \frac{k_v}{2} \int_0^t |u_{tx}|_H^2 d\tau \nonumber  \\
 \leq & \frac{m}{2}  |v_{0}|_H^2 + \frac{\gamma}{2}   |u_{0xx}|_H^2 + 
 \frac{k}{4} (|u_{0x}|_H^2  + \frac{1}{2} \int_0^1 \frac{1}{|u_{0x}|^2} dx  +   \frac{1}{2} )  \label{u-ineq}\\
 & + \frac{1}{k_v} \int_0^t  (|\hat{p}|_H^2  +  |g|^2) d\tau 
 \quad   \mbox{ for } t \in [0,  T].  \nonumber 
\end{align}
Moreover, for $i = 1, 2$ let  $g_i \in L^2(0, T)$, $u_i$ be a weak solution of (AP1)$(u_0, v_0, \hat{p}, g_i)$ on  $[0,T]$ and 
$u = u_1 - u_2$
There exist a positive constant $C_*$ depending on the lower bound of $u_{ix}$ and $|u_i|_{L^{\infty}(0,T; X)}$ for $i = 1, 2$ such that 
\begin{align}
 |u_t(t)|_H^2 + |u_{xx}(t)|_H^2 + |u_x(t)|_H^2 + \int_0^t |u_{tx}(\tau)|_H^2 d\tau 
\leq C_* \int_0^t |g_1 - g_2|^2 d\tau \quad \mbox{ for } t \in [0, T].  \label{diff_P1} 
\end{align} 
\end{lem}
\begin{proof} 
Assume $v_0 \in X$, $\hat{p} \in L^2(0, T; X)$ and $g \in C([0,T])$.  
From the above argument we then have a solution $u_n$ of  
(AP1)$_n(u_{0n}, v_{0n}, \hat{p}_n, g)$  on $[0,T]$ for $n \geq n_0$. 
From \eqref{5-1-x} and Lemma \ref{key_lemma} it follows that 
$\{u_n\}$ is bounded in $W^{1,\infty}(0, T; H)$,  $L^{\infty}(0,T; X)$  and  $W^{1,2}(0,T; H^1(0,1))$,  and 
there exist positive constant $\delta'$ and $M$ such that 
$\delta' \leq u_{nx} \leq M$ on $\overline{Q(T)}$ for $n \geq n_0$. Hence, we can take a subsequence $\{n_j\}$ 
and $u$ satisfying (W1), (W2) such that 
$u_{n_j} \to u$ weakly* in $W^{1,\infty}(0, T; H)$ and $L^{\infty}(0, T; X)$, weakly in $W^{1,2}(0, T; H^1(0,1))$, and 
$u_{n_jx} \to u_x$ in $C(\overline{Q(T)})$ as $j \to \infty$. By letting $\eta \in W^{1,2}(0, T; X)$ with $\eta(T) = 0$, we choose 
a sequence $\{\eta_n\} \subset W^{1,2}(0, T; X)$ such that $\eta_n \in X_n$ on $[0,T]$ and $\eta_n(T) = 0$ for $n$,  and 
$\eta_n \to \eta$ in $W^{1,2}(0, T; X)$.  Hence, we can substitute $\eta_n(t)$ into \eqref{ODE1} as $z_i$,  integrate it over $[0,T]$, and 
then we can get \eqref{def_AP1} for $\eta \in W^{1,2}(0, T; X)$ with $\eta(T) = 0$, easily.  Next, let $\eta \in W^{1,2}(0, T; H) \cap L^{\infty}(0,T; X)$ with $\eta(T) = 0$ and take a sequence  $\{ \eta_j \} \subset W^{1,2}(0, T; X)$ such that 
$\eta_j \to \eta$ in $W^{1,2}(0, T; H)$ and $L^{2}(0,T; X)$ as $j \to \infty$.  Hence, it is easy to see that  \eqref{def_AP1} holds
for $\eta \in W^{1,2}(0, T; H) \cap L^{\infty}(0,T; X)$ with $\eta(T) = 0$. Also, by \eqref{5-1-x} and these convergences we obtain \eqref{u-ineq}. 

Finally, let $v_0 \in H$,  $\hat{p} \in L^2(0, T; H)$ and $g \in L^2(0,T)$.  
Since we can choose sequences $\{v_{0i}\} \subset X$, $\{\hat{p}_i\} \subset  L^2(0, T; X)$ and $\{g_i\} \subset C([0,T])$ 
such that $v_{0i} \to v_0$ in $H$, $\hat{p}_i \to \hat{p}$ in $L^2(0, T; H)$ and $g_i \to g$ in $L^2(0,T)$
as $i \to \infty$, (AP1)$(u_0, v_{0i}, \hat{p}_i, g_i)$ has a weak solution $u_i$ for each $i$. Thanks to \eqref{u-ineq},  we see that 
$\{u_i\}$ is bounded in  $W^{1,\infty}(0, T; H)$,  $L^{\infty}(0,T; X)$  and  $W^{1,2}(0,T; H^1(0,1))$,  and 
there exist positive constant $\delta''$ and $M'$ such that 
$\delta'' \leq u_{ix} \leq M'$ on $\overline{Q(T)}$ for $i$. Hence, in a similar way we can show that (AP1)$(u_0, v_{0}, \hat{p}, g)$ has a weak solution $u$ and  \eqref{u-ineq} holds for it. 

The uniqueness of weak solutions of (AP1) is a direct consequence of Proposition \ref{uni_weak_P1}.  
Also, \eqref{diff_P1} is easily obtained from  \eqref{5-1-x} and convergence of $\{u_n\}$. 
\end{proof}

\begin{proof}[Proof of Proposition \ref{exist_weak_P1}] 
Thanks to Lemma \ref{AP1} for any $\hat{u} \in L^2(0,T)$, a solution $u$ of \\ (AP1)$(u_0, v_0, \nu({\hat{p}}),  -\varphi(\hat{u}))$ on $[0,T]$ exists. 
Hence,  we  can define a mapping $\Lambda:  L^2(0,T) \to L^2(0,T)$ by $\Lambda \hat{u} = u(\cdot,1)$. 
In order to show that $\Lambda$ has a fixed point we put 
$$ K(M, T) = \{z \in L^2(0, T) :  |z|_{L^2(0,T)}^2 \leq M\} \mbox{ for } M  >  0. $$

Firstly, we show that $\Lambda: K(M, T) \to K(M, T)$ for some $M > 0$. 
Let $M > 0$, $\hat{u} \in K(M,  T)$ and $u$ be a weak solution of (AP1)$(u_0, v_0, \nu({\hat{p}}),  -\varphi(\hat{u}))$ on $[0,T]$. 
By \eqref{u-ineq} in Lemma \ref{AP1} we see  that 
{
$$ |u_x(t)|_H^2 \leq  C_1 \int_0^t (|\varphi(\hat{u})|^2 + |\nu(\bar{p})|_H^2) d\tau + C_2 \mbox{ for } t \in [0,T], $$
where $C_1 =\frac{8}{k k_v}$ and $C_2$ is a positive constant determined by initial functions $u_0$ and $v_0$.  Accordingly, because of boundedness of $\varphi$ (see (A1)) we can take a positive constant $M$ such that 
$$ \left( C_1 \int_0^T (|\varphi(\hat{u})|^2 + |\nu(\bar{p})|_H^2) d\tau + C_2 \right) T \leq M. $$
Immediately, we have 
$$  \int_0^T |\Lambda \hat{u}(t,1)|^2 dt \leq M \quad \mbox{ for }  \hat{u} \in K(M,  T). $$
Thus, we obtain that  $\Lambda: K(M, T) \to K(M, T)$.  
 }
  
Next, we show that $\Lambda^l$ is the contraction mapping on $K(M,T)$ for some positive integer $l$. 
Indeed, let $\hat{u}_i \in K(M, T)$ and $u_i$ be a weak solution of  (AP1)$(u_0, v_0, \nu(\hat{p}),  -\varphi(\hat{u}_i))$:  on $[0,T]$ for $i = 1,  2$.  Also, we put 
$u = u_1 -  u_2$ and $\hat{u} = \hat{u}_1- \hat{u}_2$.
By \eqref{diff_P1} in Lemma \ref{AP1} we see  that 
\begin{align*}
 |u_t(t)|_H^2 + |u_{xx}(t)|_H^2 + |u_x(t)|_H^2 + \int_0^t |u_{tx}(\tau)|_H^2 d\tau 
\leq C_3 \int_0^t |\hat{u}|^2 d\tau \quad \mbox{ for } t \in [0, T], 
\end{align*}
$$ |\Lambda \hat{u}_1(t) - \Lambda \hat{u}_2(t)|^2 \leq |u_x(t)|_H^2 \leq  C_3 \int_0^t |\hat{u}|^2 d\tau \quad \mbox{ for } t \in [0, T],  \hat{u}_i \in K(M, T), i= 1, 2, $$ 
where $C_3$ is a positive constant. 
Hence, we see that 
$$ |\Lambda \hat{u}_1(t) - \Lambda \hat{u}_2(t)|^2 \leq  C_3 |\hat{u}|_{L^2(0,t)}^2, 
 |\Lambda \hat{u}_1- \Lambda \hat{u}_2|_{L^2(0,t)} ^2 \leq  C_3 |\hat{u}|_{L^2(0,t)}^2 t \mbox{ for } t \in [0,T], $$
and recursively, 
$$ |\Lambda ^l\hat{u}_1- \Lambda^l \hat{u}_2|_{L^2(0,t)} ^2 \leq \frac{ (C_3 t)^l}{l!}  |\hat{u}|_{L^2(0,t)}^2 \mbox{ for } t \in [0,T] 
\mbox{ and } l \geq 1.  $$
Therefore, for large $l$ since $\Lambda^l$ is a contraction on $K(M, T)$ so that (P1)($\hat{p}, u_0,  v_0$) has a weak solution. 
\end{proof} 

\vskip 12pt
In the rest of this section, we consider strong solutions to (P1)($\hat{p}, u_0,  v_0$). 
{
First, we solve the following linear problem (LP1)($u_0, v_0, g, q$):
\begin{align*}
& m u_{tt}  + \gamma u_{xxxx} -  k' u_{xx} -   k_v u_{txx} =  g_{x}   \mbox{ in } Q(T),  \nonumber \\
& u(t, 0) = 0, u_{xx}(t,0) = u_{xx}(t,1) = 0 \mbox{ for } t \in [0,T],   \nonumber \\
&  - \gamma u_{xxx}(t,1) +  k' u_x(t,1) + k_v u_{tx}(t,1) + g(t, 1)  = q(t) \mbox{ for } t \in [0,T],    \nonumber \\
&  u(0) = u_{0}, u_t(0) = v_{0}, \nonumber
\end{align*}
where $k' = \frac{k}{2}$. 
}

\begin{lem} \label{strongAP1}

Assume (A2),  $u_0 \in V \cap X$ with $u_{0x} > 0$ on $[0,1]$,  $u_{0xx}(0)  = u_{xx}(1) = 0$ and $v_0 \in X$. 

(1) If $g \in L^{\infty}(0,T; H^1(0,1))$  $\cap W^{1,2}(0,T; H)$ with $g(0) \in H^1(0,1)$, $q \in W^{1,2}(0,T)$
and $-\gamma u_{0xxx}(1) + k' u_{0x}(1) +  k_v v_{0x}(1) + g(0,1) = q(0)$, 
then there exists one and only one strong solution $u \in S_1(T)$ of   (LP1)$(u_0, v_0, g, q)$ on $[0,T]$. 
Moreover, for any $M > 0$ there exists a positive constant $C_*(M)$ such that 
if $|g|_{ L^{\infty}(0,T; H^1(0,1))} +  |g|_{W^{1,2}(0,T; H)} + |q|_{W^{1,2}(0,T)} \leq M$, then
$$ |u|_{W^{2,\infty}(0,T; H)} +  |u|_{W^{1,\infty}(0,T; X)} +  |u|_{W^{2,2}(0,T; H^1(0,1))} +  |u|_{L^{\infty}(0, T;  V)}  \leq C_*(M). $$ 

(2) Let  $\{g_n\}$ and $\{ q_n\}$ be bounded  sets in  $L^{\infty}(0,T; H^1(0,1)) \cap W^{1,2}(0,T; H)$, and $W^{1,2}(0,T)$, respectively. 
 Assume $-\gamma u_{0xxx}(1) + k' u_{0x}(1) +  k_v v_{0x}(1) + g_n(0,1) = q_n(0)$ and denote by $u_n$  a solution of (LP1)$(u_0, v_0, g_n, q_n)$ for each $n$. 
If $g_n \to g$ in $L^2(0, T; H^1(0,1)$ and  $q_n \to q$ in  $L^{2}(0,T)$ as $n \to \infty$, then 
$u_n \to u$  weakly* in $W^{2,\infty}(0,T; H), W^{1,\infty}(0,T; X), L^{\infty}(0, T; V)$  and weakly in $W^{2,2}(0, T; H^1(0,1))$ as $n \to \infty$. 
\end{lem} 
\color{black}

Since the problem \eqref{5m-1} $\sim$ \eqref{5m-3} is linear, we can easily prove this property. Here, 
for a proof of this result we quote  \cite{Bro-Sp}.  Also, the uniqueness of weak solutions to this problem holds because of  Proposition \ref{uni_weak_P1}.

\begin{prop} \label{exist_strong_P1}
Assume (A2) and (A3). 
If $u_0 \in V$ with $u_{0x} > 0$ on $[0,1]$ and $u_{0xx}(0)  = u_{xx}(1) = 0$,  $v_0 \in X$, $\hat{p} \in L^{\infty}(0,T; H^1(0,1)) \cap W^{1,2}(0,T; H)$ and 
$- \gamma u_{0xxx}(1) + f(u_{0x}(1)) + k_v v_{0x}(1) + \nu(\hat{p}(0,1)) + \varphi(s_0) =0$,  then 
 (P1)($\hat{p}, u_0,  v_0$) has  one and only one solution  $u$ on $Q(T)$ such that 
 $u \in S_1(T)$,  $u_x > 0$ on $\overline{Q(T)}$ 
 and \eqref{EQ1}, \eqref{BC1-1}, \eqref{BC1-2} and \eqref{IC1} hold  in the usual sense. 
\end{prop} 

\begin{proof}
Proposition \ref{exist_weak_P1} implies existence of a weak solution $u$ of (P1)($\hat{p}, u_0,  v_0$). It is clear that 
$\displaystyle  - \frac{k}{4} (1 + \frac{1}{(u_x)^3}) := g \in L^{\infty}(0,T; H^1(0,1)) \cap W^{1,2}(0,T; H)$ and $\varphi(s) \in W^{1,2}(0,T)$, and then  by Lemma \ref{strongAP1} there exits a strong solution $\bar{u}$ of (LP1)$(u_0, v_0, \nu(\hat{p}) + g,  - \varphi(s))$. On the other hand, $u$ is a weak solution of (LP1)$(u_0, v_0, \nu(\hat{p}) + g,  -\varphi(s))$ on $[0,T]$. 
According to the uniqueness of solutions of (LP1)$(u_0, v_0, \nu(\hat{p}) + g,  -\varphi(s) )$, we get $u = \bar{u}$.  
Thus, this shows the conclusion of  this proposition. 
\end{proof} 
\color{black} 

\section{ Solvability of (P2)}  \label{P2} 

First, we show that  (P2)$(\hat{v}, s, h_0, p_0)$ and  ($\overline{\mbox{P2}})(\bar{v}, s, h_0, \bar{p}_0$) are equivalent. 

\begin{lem} \label{equi_P2}
Assume (A1), (A2),  $s \in W^{1,2}(0,T)$,  { $s > 0$ on $[0,T]$}, $u \in C^1(\overline{Q(T)})$ with $u_x > 0$ on  $\overline{Q(T)}$ and 
$v \in L^{\infty}(0,T; H^1(0,1))$ with $v(t,0) = 0$ for $t \in [0,T]$. Also,  let 
$\hat{v}(t, y ) = v(t, u^{-1}(t,y))$ for $(t,x) \in Q(s, T)$,  and $\bar{p}(t, x)  = p(t, s(t)x)$ and 
$\bar{v}(t, x) = \hat{v}(t, s(t)x)$ for $(t,x) \in Q(T)$. 
In this case 
$p$ satisfies (S3), \eqref{EQ2}, \eqref{BC2} and  \eqref{IC2}  if and only if $\bar{p} \in W^{1,2}(0, T; H) \cap L^2(0,T; H^2(0,1)) \cap L^{\infty}(0, T; H^1(0,1)) (=: S_2(T))$, and \eqref{EQ2'}, \eqref{BC2'} and  \eqref{IC2'} hold. 

\end{lem}
 
Hence, we call that $\bar{p}$ is a solution of ($\overline{\mbox{P2}})(\bar{v}, s, h_0, \bar{p}_0$), when $\bar{p} \in S_2(T)$, and \eqref{EQ2'}, \eqref{BC2'} and  \eqref{IC2'} hold.

\begin{prop} \label{exist_P2}
If  (A1)  and (A2)  hold,  $s \in W^{2,2}(0,T)$ with $s > 0$ on $[0,T]$ and $\bar{v} \in L^{\infty}(0,T; H^1(0,1))$, 
$h_0 \in W^{1,1}(0,T)$ and $\bar{p}_0 \in H^1(0,1)$,  then  ($\overline{\mbox{P2}})(\bar{v}, s, h_0, \bar{p}_0$) has a unique solution. 
\end{prop}

In order to prove Proposition \ref{exist_P2} we introduce the following auxiliary problem ($\overline{\mbox{AP2}})(s, f, h_0,  h_1, \bar{p}_0$):  
\begin{align}
&  \rho(\bar{p})_t  -  \frac{\kappa}{s^2}  \bar{p}_{xx} = f   \mbox{ in } Q(T),   \nonumber \\
& \frac{\kappa}{s(t)}   \bar{p}_x(t,0) = h_0(t), \frac{\kappa}{s(t)}  \bar{p}_x(t, 1) =  h_1(t)  \mbox{ for } 0 < t < T,  \nonumber \\
& \bar{p}(0,x) = \bar{p}_0(x) \mbox{ for }  x \in (0,1).  \nonumber 
\end{align} 
Since $\rho$ is bi-Lipschitz continuous, thanks to the classical theory (cf. \cite{Ke(Chiba)}, e.g.) for parabolic equations we obtain the next Lemma.

\begin{lem} \label{AP2} 
Assume (A1). 
If $s \in W^{2,2}(0,T)$ with $s > 0$ on $[0,T]$, $f \in L^2(0, T; H)$,  $h_0, h_1 \in W^{1,1}(0, T)$ and $\bar{p}_0 \in H^1(0,1)$, 
then ($\overline{\mbox{AP2}})(s, f, h_0,  h_1, \bar{p}_0$) has a unique solution $\bar{p} \in S_2(T)$, which is defined in Lemma \ref{equi_P2}. Moreover, 
the function $\displaystyle t \to \frac{\kappa}{2s(t)^2} |\bar{p}_x(t)|_H^2 + \frac{1}{s(t)}  ( h_0(t) \bar{p}(t,0) -  h_1(t) \bar{p}(t,1)$ is absolutely continuous on $[0,T]$ and it holds that 
\begin{align*}
 & - \frac{\kappa}{s(t)^2} (\bar{p}_{xx}(t), \bar{p}_t(t))_H \\
  =  &  \frac{d}{dt} \left( \frac{\kappa}{2s(t)^2} |\bar{p}_x(t)|_H^2 + \frac{1}{s(t)}  ( h_0(t) \bar{p}(t,0) -  h_1(t) \bar{p}(t,1)) \right) 
  + \frac{\kappa s '(t)}{s(t)^3} |\bar{p}_x(t)|_H^2  \\
 &  + \frac{s'(t)}{s(t)^2} ( h_0(t) \bar{p}(t,0) - h_1(t) \bar{p}(t,0) ) 
      - \frac{1}{s(t)} ( h_0'(t) \bar{p}(t,0) - h_1'(t) \bar{p}(t,1) )  \quad \mbox{ for a.e. } t \in [0,T]. 
 \end{align*}
\end{lem}
\color{black}

\begin{proof}[Proof of Proposition \ref{exist_P2}]
Under the assumption of Proposition \ref{exist_P2} uniqueness of solutions of ($\overline{\mbox{P2}})(\bar{v}, s, h_0, \bar{p}_0$)  is a direct consequence of the result proved in Section \ref{unique}. 

We shall  prove existence of a solution of (P2) by applying Banach's fixed point theorem. 
For $M_1 > 0$ we put $K_2(M_1, T) = \{ z \in L^4(0, T; H^1(0,1)):  \int_0^T |z|_{H^1(0,1)}^4 dt \leq M_1\}$.  
Let $\bar{q} \in K_2(M_1, T)$ and put 
$$ h_1  = - s' \psi(s), f = - \frac{1}{s} (\bar{v}  \rho(\bar{q}))_x + \frac{xs'}{s} \rho(\bar{q})_x.  $$
Since  $h_1 \in W^{1,2}(0, T)$ and $f \in L^2(0, T; H)$, Lemma \ref{AP2} implies existence of a solution $\bar{p} \in S_2(T)$ of ($\overline{\mbox{AP2}})(s, f, h_0,  h_1, \bar{p}_0$). 
Hence, we define the solution operator $\Lambda_2 : K_2(M_1, T)  \to L^4(0, T; H^1(0,1))$ by $\Lambda_2 \bar{q} = \bar{p}$. 

First, we show that for a large positive integer $l$ and some $M_1 > 0$, $\Lambda_2^l : K_2(M_1, T)  \to  K_2(M_1, T)$.   
Let $\bar{q} \in K_2(M_1, T)$ and $\Lambda_2 \bar{q} = \bar{p}$.  We multiply the equation by $\bar{p}_t$ and then by Lemma \ref{AP2}  and (A1) we have 
\begin{align}
& \mu |\bar{p}_t(t)|_H^2 + \frac{d}{dt} E_0(t)  \nonumber \\
\leq & - \frac{\kappa s'(t)}{s(t)^3} |\bar{p}_x(t)|_H^2 -  \frac{s'(t)}{s(t)^2} ( h_0(t) \bar{p}(t,0) - h_1(t) \bar{p}(t,1) ) 
      + \frac{1}{s(t)} ( h_0'(t) \bar{p}(t,0) - h_1'(t) \bar{p}(t,1) )  \label{eq6-113}  \\
      & + (f_1(t) \rho(\bar{q})(t), \bar{p}_t(t))_H + (f_2(t) \rho(\bar{q})_x(t), \bar{p}_t(t))_H  (=: \sum_{i=1}^5 I_i(t) ) 
       \mbox{ for a.e. } t \in [0,T],  \nonumber 
      \end{align}
where $\displaystyle E_0 = \frac{\kappa}{2s^2} |\bar{p}_x|_H^2 + \frac{1}{s}  ( h_0 \bar{p}(\cdot, 0) -  h_1 \bar{p}(\cdot,1 ))$, 
$\displaystyle f_1 = -\frac{1}{s} \bar{v}_x$ and $\displaystyle f_2 = -\frac{1}{s} \bar{v} + \frac{xs'}{s}$. It is clear that $f_1 \in L^{\infty}(0,T; H)$, $f_2 \in L^{\infty}(Q(T))$ and there exists a positive constants $C_1$ depending only on $s$,  $h_0$ and $h_1$  such that 
$$ E_0 \geq \frac{\kappa}{4 R^2} |\bar{p}_x|_H^2  - C_1(|\bar{p}|_H^2 + 1) \mbox{ on } [0,T],  $$
where $R = \sup\{s(t):  t \in [0,T]\}$. By putting $E_1 = E_0 + C_1(|\bar{p}|_H^2 + 1) + |\bar{p}|_H^2$, we see that 
$$ E_1 \geq  |\bar{p}|_H^2, \frac{4R^2}{\kappa}E_1 \geq  |\bar{p}_x|_H^2  \mbox{ on } [0,T]. $$
From these inequalities it follows that 
$$ I_1 \leq C_2  E_1,  I_2 \leq C_2 (E_1 +1),  I_3 \leq C_2 (|h_0'| + |h_1'|)  (E_1 +1), $$ 
$$ I_4 + I_5 \leq \frac{\mu}{2} |\bar{p}_t|_H^2 + C_2 (|\bar{q}|_{H^1(0,1)}^2 + 1) \mbox{ on } [0,T], $$
 where $C_2$ is a positive constant depending only on $s$, $h_0$ and $\bar{v}$.  Hence, we have 
 \begin{align*}
 & \frac{\mu}{2} |\bar{p}_t|_H^2 + \frac{d}{dt} (E_1 + 1) \\
 \leq & 2C_2(1+ |h_0'| + |h_1'|)(E_1 + 1) +  C_2 |\bar{q}|_{H^1(0,1)}^2 + (C_1+1) |\bar{p}|_H  |\bar{p}_t|_H
 \end{align*}
and 
\begin{equation}
 \frac{\mu}{4} |\bar{p}_t|_H^2 + \frac{d}{dt} (E_1 + 1) 
 \leq C_3(1 + |h_0'| + |h_1'|)(E_1 + 1) +  C_2 |\bar{q}|_{H^1(0,1)}^2  \mbox{ a.e. on }[0,T], \label{f1-2}
\end{equation} 
where $C_3$ is a positive constant. By applying Gronwall's inequality, we infer that 
\begin{equation}
\int_0^t |(\Lambda_2\bar{q})_t|_H^2 d\tau + |(\Lambda_2\bar{q})(t)|_{H^1(0,1)}^2 \leq C_4 \int_0^t |\bar{q}|_{H^1(0,1)}^2 d\tau + C_4
 \mbox{ for } t \in [0,T],  \bar{q} \in K_2(M_1, T),  \label{eq6-83}
 \end{equation} 
where $C_4$ is a positive constat. 
Similarly to the proof of  Proposition \ref{exist_weak_P1},  we observe that 
$$ |\Lambda_2^l \bar{q}(t)|_{H^1(0,1)}^2 \leq  \frac{ C_4 (C_4t)^{l-1}}{(l -1)!} |\bar{q}|_{L^2(0,T; H^1(0,1))}^2 + C_4 e^{C_4t} 
      \mbox{ for } t \in [0,T],  l = 1, 2, \cdots, \bar{q} \in K_2(M_1, T).  $$
Immediately, we get 
$$ \int_0^T |\Lambda_2^l \bar{q}|_{H^1(0,1)}^4 dt \leq  \left(  \frac{ C_4 (C_4T)^{l-1}}{(l -1)!} \sqrt{M_1} +   C_4 e^{C_4T} \right)^2 T 
          \quad     \mbox{ for }  l = 1, 2, \cdots \mbox{ and } \bar{q} \in K_2(M_1, T).    $$
Hence, we can take a positive integer $l_0$  and $M_1 > 0$ such that 
\begin{equation}
\int_0^T |\lambda^l \bar{q}|_{H^1(0,1)}^4  dt \leq M_1 
              \mbox{ for }  l \geq l_0 \mbox{ and } \bar{q} \in K_2(M_1, T).   \label{eq6-6.4}
\end{equation}
Moreover,  by \eqref{eq6-83} and \eqref{eq6-6.4}  for some $M_2 > 0$ it holds that 

\begin{equation}
|\Lambda_2^{l}  \bar{q}|_{W^{1,2}(0,T; H)}^2 +  |\Lambda_2^{l} \bar{q}|_{L^{\infty}(0,T; H^1(0,1))}^4  \leq M_2  \mbox{ for }   \bar{q} \in K_2(M_1, T) \mbox{ and } 
 l = 1, 2, \cdots. 
\label{eq6-24}
\end{equation}

Next, we show that $\Lambda_2^l$ is a contraction on $K_2(M_1, T)$ for some $l \geq l_0$.  Let $\bar{q}_i \in K_2(M_1, T)$ and $\bar{p}_i = \Lambda_2\bar{q}_i$ for $i = 1, 2$. For simplicity we put $\bar{p} = \bar{p}_1 - \bar{p}_2$ and  $\bar{q} = \bar{q}_1 - \bar{q}_2$, and then we have
$$ \rho(\bar{p}_1)_t - \rho(\bar{p}_{2})_t - \frac{\kappa}{s^2} \bar{p}_{xx} = f_1 ( \rho(\bar{q}_{1}) - \rho(\bar{q}_{2})) 
+ f_2 ( \rho(\bar{q}_{1})_x - \rho(\bar{q}_{2})_x)  \mbox{ a.e. on } Q(T). 
$$
Similarly to \eqref{eq6-113}, by multiplying it with $\bar{p}_t$ and Lemma \ref{AP2}, we see that 
\begin{align*}
  & \ ( \rho(\bar{p}_1)_t - \rho(\bar{p}_{2})_t, \bar{p}_t)_H + \frac{d}{dt} (\frac{\kappa}{2s} |\bar{p}_x|_H^2) + \frac{\kappa s'}{2s^3} |\bar{p}_x|_H^2 \\
 = & \ ( f_1 ( \rho(\bar{q}_{1}) - \rho(\bar{q}_{2})),  \bar{p}_t)_H 
+ (f_2 ( \rho(\bar{q}_{1})_x - \rho(\bar{q}_{2})_x),  \bar{p}_t)_H \quad \mbox{ a.e. on } [0,T]. 
\end{align*}
Here, on account of Lemma \ref{lemma1} we note that 

\begin{align*}
 ( \rho(\bar{p}_1)_t - \rho(\bar{p}_{2})_t, \bar{p}_t)_H 
=  &  ( \rho'(\bar{p}_1) \bar{p}_t, \bar{p}_t)_H + ((\rho'(\bar{p}_1) -  \rho'(\bar{p}_2)) \bar{p}_{2t}, \bar{p}_t)_H \\
\geq & \mu |\bar{p}_t|_H^2 - C_{\rho} |\bar{p}|_{L^{\infty}(0,1)}   |\bar{p}_{2t}|_H |\bar{p}_t|_H \\
\geq & \frac{\mu}{2}  |\bar{p}_t|_H^2 - C_5 (|\bar{p}|_{H}^2 + |\bar{p}_x|_{H}^2)   |\bar{p}_{2t}|_H^2, 
\end{align*}
$$
 (  f_1 ( \rho(\bar{q}_{1}) - \rho(\bar{q}_{2})),   \bar{p}_t)_H 
\leq  \frac{\mu}{8}  |\bar{p}_t|_H^2 + C_5 |\bar{q}|_{H^1(0,1)}^2  |f_1|_H^2,  
$$
$$
 (  f_2 ( \rho(\bar{q}_{1})_x - \rho(\bar{q}_{2})_x),    \bar{p}_t)_H 
\leq  \frac{\mu}{8}  |\bar{p}_t|_H^2 + C_5   (|\bar{q}_1|_{H^1(0,1)}^2 + 1)  |f_2|_{L^{\infty}(Q(T))}^2 |\bar{q}|_{H^1(0,1)}^2
\mbox{ a.e. on } (0,T), 
$$
where $C_5$ is a positive constant.  \color{black}
Hence, we have
\begin{align*}
 & \frac{\mu}{4} |\bar{p}_t|_H^2 + \frac{1}{2}  \frac{d}{dt} ( \frac{\kappa}{s^2} |\bar{p}_x|_H^2) \\
 \leq & C_6 |\bar{p}_x|_H^2 + C_6   |\bar{p}_{2t}|_H^2 (|\bar{p}|_{H}^2 +  |\bar{p}_x|_{H}^2) 
   + C_6 (1 +  |\bar{q}_1|_{H^1(0,1)}^2)  |\bar{q}|_{H^1(0,1)}^2 \mbox{ a.e. on } [0,T], 
\end{align*}
where $C_6$ is a positive constant. By putting $E_2 = \frac{\kappa}{2s^2} |\bar{p}_x|_H^2 + |\bar{p}|_H^2$, similarly to \eqref{f1-2}, we obtain
$$  \frac{\mu}{8} |\bar{p}_t|_H^2 +  \frac{d}{dt} E_2 
 \leq  C_7 (|\bar{p}_{2t}|_H^2 +1) E_2 
 + C_7  (1 +  |\bar{q}_1|_{H^1(0,1)}^2)  |\bar{q}|_{H^1(0,1)}^2 \mbox{ a.e. on } [0,T], 
$$ 
where $C_7$ is a positive constant. Consequently, by applying Gronwall's inequality,  we infer that 
\begin{align*}
 E_2(t)  & \leq C_7 \exp( C_7 \int_0^t (|\bar{p}_{2t}|_H^2 +1)d\tau) (\int_0^t  (1 +  |\bar{q}_1|_{H^1(0,1)}^2)  |\bar{q}|_{H^1(0,1)}^2 d\tau) \\
& \leq C_7 \exp( C_7 \int_0^t (  |\bar{p}_{2t}|_H^2 +1)d\tau)  (\int_0^t  (1 +  |\bar{q}_1|_{H^1(0,1)}^2)^2 d\tau)^{1/2} 
      (\int_0^t  |\bar{q}|_{H^1(0,1)}^4 d\tau)^{1/2}  
\end{align*}
and thanks to \eqref{eq6-24} we have
$$ |(\Lambda_2^l \bar{q}_1 - \Lambda_2^l \bar{q}_2)(t)|_{H^1(0,1)}^4  \leq C_8   \int_0^t |\Lambda_2^{l-1} \bar{q}_1- \Lambda_2^{l-1}\bar{q}_2|_{H^1(0,1)}^4 d\tau 
 \mbox{ for } t \in [0,T], \bar{q}_1, \bar{q}_2 \in K_2(M_1, T),  
 $$
and $l \geq 1$,  where $C_8$ is a positive constant depending on $M_1$ and $M_2$. It is easy to see that 
 $$ \int_0^t |\Lambda_2^l \bar{q}_1 - \Lambda_2^l \bar{q}_1|_{H^1(0,1)}^4 d\tau \leq \frac{ (C_8  t)^l}{l!}  \int_0^t | \bar{q}_1 - \bar{q}_2|_{H^1(0,1)}^4 d\tau 
 \mbox{ for } t \in [0,T], \bar{q} \in K_2(M_1, T), l = 1, 2, \cdots. 
 $$
Thus, we can take $l \geq l_0$ such that $\Lambda_2^l$ is the contraction mapping on $K_2(M_1, T)$. 
Therefore, we have proved existence of a solution to (P2). 
\end{proof}

\section{Existence of solutions to problem (P)}
Throughout this section we suppose that all assumptions as in Theorem \ref{main_th} hold.\\

Let $\hat{p} \in  S(T) := \{ z \in S_2(T) : z(0,1) = p_0(s_0) \}$. 
Proposition \ref{exist_strong_P1} implies existence of a solution $u \in S_1(T)$ of  (P1)($\hat{p}, u_0,  v_0$). 
Here, we put $s  = u(\cdot,1)$ and $\bar{v}(t,x)  = u_t(t, u^{-1}(t, s(t)x))$ for $(t,x) \in Q(T)$. It is clear that $s \in W^{2,2}(0, T)$ with $s  > 0$ on $[0, T]$ and 
$\bar{v} \in L^{\infty}(0, T; H^1(0,1))$. Accordingly, by  Proposition \ref{exist_P2} we obtain a unique solution $\bar{p} \in S_2(T)$ of  ($\overline{\mbox{P2}})(\bar{v}, s, h_0, \bar{p}_0$).  
Hence, we can define a mapping $\Gamma : S(T)  \to S(T)$ by $\Gamma \hat{p} = \bar{p}$. 
Moreover, we put $\Gamma^{(1)} \hat{p} = u$,    $\Gamma^{(2)} \hat{p} = s$ and $\Gamma^{(3)} \hat{p} = \bar{v}$ for any $\hat{p} \in  L^2(0, T; H^1(0,1))$. 
We prove existence of solutions of (P) by applying Schauder's fixed point theorem to $\Gamma$ in the weak topology of $S(T)$.\\

As a first step of the proof we provide the following Lemma. 
\begin{lem} \label{lem7-1}
The set $\{\Gamma^{(1)} \hat{p}: \hat{p} \in  S(T) \}$ is bounded in {$S_1(T)$. }
Moreover,   there exists positive constants $\delta$ and $M$ such that 
\begin{equation}
\delta \leq (\Gamma^{(1)} \hat{p})_x \leq M \mbox{ on } \overline{Q(T)} \mbox{ for } \hat{p} \in S(T).  \label{eq7-100}
\end{equation}
\end{lem}
\begin{proof}
Put $u = \Gamma^{(1)} \hat{p}$ and $s = \Gamma^{(2)} \hat{p}$. 
By multiplying  \eqref{EQ1} with $u_t$ and integrating it,  we wee that 
\begin{align}
& \frac{d}{dt} \left( \frac{m}{2}  |u_{t}|_H^2 + \frac{\gamma}{2}   |u_{xx}|_H^2 +
 \frac{k}{4} (|u_{x}|_H^2 - |u_{x}|_{L^1(0,1)} ) + \frac{k}{8} \int_0^1 \frac{1}{|u_{x}|^{2}} dx \right) +  k_v |u_{tx}|_H^2  \nonumber  \\
 \leq & |\nu(\hat{p})|_H |u_{tx}|_H - s' \varphi(s) \quad   \mbox{  on  }  [0, T],  \nonumber 
\end{align} 
and 
\begin{align}
& \frac{d}{dt} \left( \frac{m}{2}  |u_{t}|_H^2 + \frac{\gamma}{2}   |u_{xx}|_H^2 +
 \frac{k}{4} (|u_{x}|_H^2 - |u_{x}|_{L^1(0,1)}) + \frac{k}{8} \int_0^1 \frac{1}{|u_{x}|^{2}} dx + \varphi(s) \right) + \frac{ k_v}{2}  |u_{tx}|_H^2  \nonumber  \\
 \leq & \frac{1}{2k_v} |\nu(\hat{p})|_H^2   \quad   \mbox{  on  }  [0, T]. \nonumber 
\end{align} 
Hence, thanks to the assumption (A2) and (A3) we obtain the boundedness of the set $\{\Gamma^{(1)} \hat{p} : \hat{p} \in  S(T)  \}$ in $W_1(T)$.

Accordingly,  by Lemma \ref{key_lemma}  \eqref{eq7-100} holds for some $\delta > 0$ and $M > 0$. Moreover, we can take $M' > 0$ such that
$$ \left| \frac{1}{(\Gamma^{(1)} \hat{p})_x^3}\right|_{L^{\infty}(0, T; H^1(0,1))}  + \left| \frac{1}{(\Gamma^{(1)} \hat{p})_x^3}\right|_{ W^{1,2}(0, T; H^1(0,1))} 
+ | \varphi( \Gamma^{(2)} \hat{p} )|_{W^{1,2}(0,T)} \leq M'   \quad \mbox{ for } \hat{p} \in  S(T). 
$$
Clearly, $\Gamma^{(1)} \hat{p}$ is a solution of  
(LP1)$\displaystyle (u_0, v_0, -\frac{\kappa}{4}(1 + \frac{1}{ ( (\Gamma^{(1)} \hat{p})_x)^3}) + \nu(\hat{p}), - \varphi( \Gamma^{(2)} \hat{p}) )$
on $[0,T]$.  Therefore, Lemma \ref{strongAP1} implies the boundedness of $\{\Gamma^{(1)} \hat{p}: \hat{p} \in  S(T) \}$ in $S_1(T)$. 

\end{proof}

Next, we show continuity of  the mapping $\Gamma$ by the following Lemma.  
\begin{lem} \label{lem7-3}
If $\{\hat{p}_n \}$ is bounded in $W^{1,2}(0,T; H)$ and $L^{\infty}(0, T; H^1(0,1))$ and 
$\hat{p}_n \to \hat{p}$ in $L^2(0, T; H^1(0,1))$ as $n \to \infty$,  
then $\Gamma^{(1)} \hat{p}_n \to \Gamma^{(1)} \hat{p}$ in $W_1(T)$,  weakly* in $W^{2,\infty}(0,T; H)$, $W^{1,\infty}(0,T; X)$, $L^{\infty}(0, T; V)$  and weakly in $W^{2,2}(0, T; H^1(0,1))$ as $n \to \infty$,  
$\Gamma^{(2)} \hat{p}_n \to \Gamma^{(2)} \hat{p}$ in $W^{2,2}(0,T)$, $\{\Gamma^{(3)} \hat{p}_n \}$ is bounded  in $L^{\infty}(0, T; H^1(0,1))$ and  
$\Gamma^{(3)} \hat{p}_n \to \Gamma^{(3)} \hat{p} $ in $L^{2}(0, T; H^1(0,1))$ as $n \to \infty$. 
\end{lem}
\begin{proof}
Let $u_n = \Gamma^{(1)} \hat{p}_n$ for $n$ and $u = \Gamma^{(1)} \hat{p}$, 
 $s_n = \Gamma^{(2)} \hat{p}_n $ on $[0,T]$ for $n= 1, 2, \cdots$ and $s = \Gamma^{(2)} \hat{p} $ on $[0,T]$, 
Similarly to \eqref{4-i}, we see that 
\begin{align}
& \frac{d}{dt} ( \frac{m}{2} |u_{nt}(t) - u_t(t)|_H^2 + \frac{\gamma}{2} |u_{nxx}(t)  - u_{xx}(t)|_H^2  
+ \frac{k}{4} |u_{nx}(t) - u_x(t)|_H^2) + \frac{k_v}{2}  |u_{ntx}(t) - u_{tx}(t)|_H^2 \nonumber \\
 \leq & C_1 |u_{nx}(t) - u_x(t)|_H^2  + C_1 |\hat{p}_n(t) - \hat{p}(t)|_H \mbox{ for } t \in [0,T],   \nonumber
\end{align}
where $C_1$ is a positive constant. Hence, by applying Gronwall's inequality we see that 
$u_n \to u$ in $W_1(T)$ as $n \to \infty$, since $\hat{p}_n \to \hat{p}$ in $L^2(0, T; H)$. 
Immediately, we have $\varphi(s_n) \to \varphi(s)$ in $W^{1,2}(0,T)$ and 
$$  \frac{\kappa}{4} (1 + \frac{1}{u_{nx}^3} ) \to  \frac{\kappa}{4} (1 + \frac{1}{ u_{x}^3} ) 
  \mbox{ in } L^{\infty}(0,T; H^1(0,1)), W^{1,2}(0, T; H) \mbox{ as } n \to \infty. $$ 
It is clear that $u_n$  is a solution of (LP1)$\displaystyle (u_0, v_0, -\frac{\kappa}{4}(1 + \frac{1}{ ( u_{nx})^3})  + \nu(\hat{p}_n), - \varphi(s_n) )$
on $[0,T]$ for each $n$ and $u$ is a solution of (LP1)$\displaystyle (u_0, v_0, -\frac{\kappa}{4}(1 + \frac{1}{ ( u_{x})^3})  + \nu(\hat{p}), - \varphi(s) )$ on $[0,T]$. 
 Therefore, on account of Lemma \ref{strongAP1} we infer that 
 $\{u_n\}$ is bounded  in $W^{2,\infty}(0,T; H)$, $W^{1,\infty}(0,T; X)$, $L^{\infty}(0, T; V)$ and $W^{2,2}(0, T; H^1(0,1))$ as $n \to \infty$.   
Also, observe that  
 $u_n \to u$ weakly* in $W^{2,\infty}(0,T; H)$, $W^{1,\infty}(0,T; X)$, $L^{\infty}(0, T; V)$  and weakly in $W^{2,2}(0, T; H^1(0,1))$ as $n \to \infty$.   
Immediately, this shows that $s_n \to s$ weakly in $W^{2,2}(0,T)$ as $n \to \infty$. 

Next, we put  $\Gamma^{(3)} \hat{p}  =\bar{v}$  and $\Gamma^{(3)} \hat{p}_n  =\bar{v}_{n}$ for each $n$, 
Similarly to  Lemma \ref{bar_v},    for some  positive constant $C_2$ it holds that 
$$ |\bar{v}_n(t) - \bar{v}(t)|_H \leq C_2( |u_{nx}(t) - u_{x}(t)|_H + |u_{nt}(t) - u_{t}(t)|_H) \mbox{ for } 
 t \in [0, T] \mbox{ and } n = 1, 2, \cdots. $$ 
Hence, we have $\bar{v}_n \to \bar{v}$ in $L^{2}(0, T; H)$ as $n \to \infty$.

As a next step,  we show that 
$\bar{v}_{nx} \to \bar{v}_x$ in $L^{2}(0, T; H)$ as $n \to \infty$. 
First, we see that 
$$ \bar{v}_{nx}(t,x) = \frac{s_n(t) u_{ntx}(t, u_n^{-1}(t, s_n(t)x))}{u_{nx}(t, u_n^{-1}(t, s_n(t)x))}, 
\bar{v}_{x}(t,x) = \frac{s(t) u_{tx}(t, u^{-1}(t, s_n(t)x))}{u_{x}(t, u^{-1}(t, s(t)x))} \mbox{ for } (t, x) \in Q(T). 
$$ 
Immediately, we have
\begin{align*}
& |\bar{v}_{nx}(t,x) - \bar{v}_{n}(t,x) | \\
\leq & |s_n(t) - s(t)|  \left|  \frac{u_{ntx}(t, u_n^{-1}(t, s_n(t)x))}{u_{nx}(t, u_n^{-1}(t, s_n(t)x))} \right| \\ 
&  + | s(t)| |u_{ntx}(t, u_n^{-1}(t, s_n(t)x))| \left| \frac{1}{u_{nx}(t, u_n^{-1}(t, s_n(t)x))} - \frac{1}{u_{x}(t, u^{-1}(t, s(t)x))} \right| \\ 
&  + \left| \frac{s(t)}{u_{x}(t, u^{-1}(t, s(t)x))} \right| |u_{ntx}(t, u_n^{-1}(t, s_n(t)x)) - u_{tx}(t, u^{-1}(t, s(t)x)) |  \\
= : & \sum_{i=1}^3 I_{in}(t,x)  \mbox{ for } (t, x) \in Q(T).
\end{align*}
From Lemma \ref{lem7-1} it follows that 
$$ |I_{1n}|_H \leq C_3 |s_n - s| \to  0 \mbox{ as } n \to \infty \mbox{ uniformly on }   [0,T], $$
where $C_3$ is a positive constant independent of $n$. 
Easily, we get 
\begin{align*}
 I_{2n}(t,x) \leq  & C_4 (|u_{nx}(t, u_n^{-1}(t, s_n(t)x)) - u_{nx}(t, u^{-1}(t, s(t)x))|  \\
& \  + |u_{nx}(t, u^{-1}(t, s(t)x)) - u_{x}(t, u^{-1}(t, s(t)x))| )  \\ 
& \  (=: I_{2,1n}(t,x) + I_{2,2n}(t,x) )  \quad   \mbox{ for } (t, x) \in Q(T), 
\end{align*}
where $C_4$  is a positive constant independent of $n$. Similarly to \eqref{40-2}, we have 
\begin{align*}
& \   I_{2,1n}(t,x)  \\
\leq  &\  C_4 |u_{nxx}(t)|_{L^{\infty}(0,1)} |u_n^{-1}(t, s_n(t)x) -  u^{-1}(t, s(t)x)|  \\ 
\leq  &\  C_4 |u_{nxx}(t)|_{L^{\infty}(0,1)} ( |u_n^{-1}(t, s_n(t)x) -  u^{-1}(t, s_n(t)x)|  + |u^{-1}(t, s_n(t)x) -  u^{-1}(t, s(t)x)| )
 \\ 
 \leq  &\  C_4 |u_{nxx}(t)|_{L^{\infty}(0,1)} ( \frac{1}{\delta}  |u_n(t, u_n^{-1}(t, s_n(t)x)) -  u(t, u^{-1}(t, s_n(t)x))|  \\
&  \ +           |(u^{-1})_x(t)|_{L^{\infty}(0,s(t))}  |s_n(t) - s(t)| ) \quad   \mbox{ for } (t, x) \in Q(T). 
\end{align*}
Here, by change of variables we observe that 
\begin{align*}
& |I_{2,1n}(t)|_H  \\  
\leq &  \frac{C_4}{\delta}  |u_{nxx}(t)|_{L^{\infty}(0,1)} ( \frac{1}{\delta}  |u_n(t) -  u(t)|_H  +     |s_n(t) - s(t)| ) \quad   \mbox{ for } t \in [0,T],
\end{align*}
and 
\begin{align*}
|I_{2,2n}(t)|_H \leq  \frac{C_4}{\delta}   |u_n(t) -  u(t)|_H  \quad   \mbox{ for } t \in [0,T]. 
\end{align*}
Hence, it holds $|I_{2n}(t)|_H \to 0$ as $n \to \infty$ uniformly on $[0,T]$.  Moreover, we have 
\begin{align*}
 I_{3n}(t,x) \leq  & C_5 (|u_{ntx}(t, u_n^{-1}(t, s_n(t)x)) - u_{tx}(t, u_n^{-1}(t, s_n(t)x))|  \\
& \  + |u_{tx}(t, u_n^{-1}(t, s_n(t)x)) - u_{tx}(t, u^{-1}(t, s(t)x))| )  \\ 
& \  (=: I_{3,1n}(t,x) + I_{3,2n}(t,x) )  \quad   \mbox{ for } (t, x) \in Q(T), 
\end{align*}
where $C_5$  is a positive constant independent of $n$. Clearly, by change of variables we see that 
\begin{equation}
|I_{3,1n}(t)|_H  \leq \frac{1}{\delta} |u_{ntx}(t)- u_{tx}(t)|_H  \quad   \mbox{ for } t \in [0,T].  \label{7-12}
\end{equation} 
{
Since $\{u_{nt}\}$ is  bounded in $W^{1,\infty}(0, T; H)$ and $L^{\infty}(0, T; X)$, by Aubin's compact theorem (see \cite{Lions}) we can show that 
$I_{3,1n} \to 0$ in $L^2(0, T; H)$ as $n \to \infty$. }
Also, by applying  \eqref{40-2}, again, we note that 
\begin{align}
 & |I_{3,2n}(t)|_H  \nonumber \\
 \leq  & C_5  (\int_0^1 |u_{txx}(t)|_H^2 |u_n^{-1}(t, s_n(t)x)) - u^{-1}(t, s(t)x))| dx )^{1/2}   \nonumber \\
  \leq  & C_5 |u_{txx}(t)|_H (\int_0^1  (|u_n^{-1}(t, s_n(t)x)) - u_n^{-1}(t, s(t)x))| + |u_n^{-1}(t, s(t)x)) - u^{-1}(t, s(t)x))|)  dx )^{1/2}  \nonumber \\
 \leq  & C_5 |u_{txx}(t)|_H (\int_0^1   |(u_n^{-1})_x(t)|_{L^{\infty}(0,s_n(t))} |s_n(t) - s(t)| dx +  \frac{1}{\delta} \int_0^1 |u_n^{-1}(t, s(t)x)) - u^{-1}(t, s(t)x))|  dx )^{1/2} \nonumber \\ 
 \leq  & \frac{C_5}{\delta}  |u_{txx}(t)|_H (  |s_n(t) - s(t)| +  \int_0^1 |u_n(t, u^{-1}(t, s(t)x))) -  u(t, u^{-1}(t, s(t)x)))|  dx )^{1/2}  \nonumber \\ 
 \leq  & \frac{C_5}{\delta}  |u_{txx}(t)|_H (  |s_n(t) - s(t)| +   \frac{1}{\delta} \int_0^1 |u_n(t, x) -  u(t, x)|  dx )^{1/2}    \mbox{ for } t \in [0,T]. \label{7-13}
\end{align}
According to \eqref{7-12} and \eqref{7-13}, we infer that $|I_{3n}(t)|_H \to 0$ as $n \to \infty$ uniformly on $[0,T]$.\\

At the end of the proof we show that $\{ \bar{v}_{n} \}$ is bounded in $L^{\infty}(0,T; H^1(0,1))$. 
By elementary calculation the following estimates are obtained: 
\begin{equation}
\left. 
\begin{array}{l}
\displaystyle  |\bar{v}_n(t)|_H \leq \frac{1}{\sqrt{\delta}} |u_{nt}(t)|_H, \\
\displaystyle |\bar{v}_{nx}(t)|_H \leq \frac{1}{\delta \sqrt{\delta}} |u_{nx}(t)|_H |u_{ntx}(t)|_H  
\end{array} 
\right\} \quad 
\mbox{ for } t \in [0,T] \mbox{ and } n = 1, 2, \cdots.  \label{eq7-4} 
\end{equation}
Thus, we have proved this Lemma. 
\end{proof}

Next, we show boundedness of  $\{\Gamma \hat{p}: \hat{p} \in S(T)\}$. 
\begin{lem} \label{lem7-4}
The set $\{ \Gamma \hat{p}: \hat{p} \in S(T)\}$ is bounded in $S_2(T)$, namely, 
there exists a positive constant $M^*$ such that 
$$ |\Gamma \hat{p}|_{W^{1,2}(0,T; H)} + |\Gamma \hat{p}|_{L^{\infty}(0,T; H^1(0,1))} + |\Gamma \hat{p}|_{L^{2}(0,T; H^2(0,1))} 
 \leq M^* \mbox{ for } \hat{p} \in S(T).  $$
\end{lem}
\begin{proof}
By Lemma \ref{lem7-1} $\{\Gamma^{(1)} \hat{p}: \hat{p} \in  S(T) \}$ is bounded in $S_1(T)$ and from \eqref{eq7-4} it follows that 
$\{\Gamma^{(3)} \hat{p}: \hat{p} \in S(T) \}$ is bounded in $L^{\infty}(0, T; H^1(0,1))$. 
Let $\hat{p} \in S(T)$, $s = \Gamma^{(2)} \hat{p}$, $\bar{v} = \Gamma^{(3)} \hat{p}$, $\bar{p} = \Gamma \hat{p}$, 
$h_1 = -\varphi(s)$, 
 $\displaystyle E_0 = \frac{\kappa}{2s^2} |\bar{p}_x|_H^2 + \frac{1}{s}  ( h_0 \bar{p}(\cdot, 0) -  h_1 \bar{p}(\cdot,1 ))$, 
$\displaystyle f_1 = -\frac{1}{s} \bar{v}_x$ and $\displaystyle f_2 = -\frac{1}{s} \bar{v} + \frac{xs'}{s}$.
Similarly to  \eqref{eq6-113} and  \eqref{eq6-83},   from  
\begin{align}
& \mu |\bar{p}_t(t)|_H^2 + \frac{d}{dt} E_0(t)  \nonumber \\
\leq & - \frac{\kappa s'(t)}{s(t)^3} |\bar{p}_x(t)|_H^2 -  \frac{s'(t)}{s(t)^2} ( h_0(t) \bar{p}(t,0) - h_1(t) \bar{p}(t,1) ) 
      + \frac{1}{s(t)} ( h_0'(t) \bar{p}(t,0) - h_1'(t) \bar{p}(t,1) )  \label{eq7-5}  \\
      & + (f_1(t) \rho(\bar{p})(t), \bar{p}_t(t))_H + (f_2(t) \rho(\bar{p})_x(t), \bar{p}_t(t))_H  \quad 
       \mbox{ for a.e. } t \in [0,T],  \nonumber 
      \end{align}
we can obtain the desired boundedness of $\{\Gamma \hat{p}: \hat{p} \in S(T)\}$. 
\end{proof}

\begin{proof}[Proof of Theorem \ref{main_th}(existence)]
First, we put  
$$ S_* = \{ z \in S(T) :   |z|_{W^{1,2}(0,T; H)} + |z|_{L^{\infty}(0,T; H^1(0,1))}  +  |z|_{L^{2}(0,T; H^2(0,1))}  \leq M^* \}$$
 for the constant  $M*$ obtained in Lemma \ref{lem7-4}. 
It is obvious that  $S_*$ is  non-empty convex, $\Gamma$ is the mapping from $S_*$ into itself, and closed in the following sense: If $z_n \in S_*$ for $n = 1, 2, \cdots$ and $z_n \to z$ weakly in $W^{1,2}(0,T; H)$ and $L^2(0, T; H^2(0,1))$, and weakly* in $L^{\infty}(0,T; H^1(0,1))$, then $z \in S_*$.  
Also, $S_*$ is compact in the same (weak) topology of $W^{1,2}(0,T; H) \cap L^{\infty}(0,T; H^1(0,1)) \cap L^2(0,T; H^2(0,1))$. 
Thus, in order to apply Schauder's fixed point theorem to the mapping $\Gamma$ it is sufficient to show that
$\hat{p}_n \in S(T)$ for $n$,  $\hat{p} \in S(T)$ and $\hat{p}_n \to \hat{p}$ weakly in $W^{1,2}(0,T; H)$ and $L^2(0, T; H^2(0,1))$, and weakly* in $L^{\infty}(0,T; H^1(0,1))$,  then $\Gamma \hat{p}_n \to \Gamma \hat{p}$ in the same topology. 

Let $\hat{p}_n \in S(T)$ for $n$,  $\hat{p} \in S(T)$ and $\hat{p}_n \to \hat{p}$ weakly in $W^{1,2}(0,T; H)$  and $L^2(0, T; H^2(0,1))$,  and weakly* in $L^{\infty}(0,T; H^1(0,1))$.  Clearly, $\{ \hat{p}_n\}$ is bounded in $W^{1,2}(0,T; H)$,  $L^2(0, T; H^2(0,1))$ and $L^{\infty}(0,T; H^1(0,1))$. 
Hence, by applying Aubin's compact theorem we can take a subsequence $\{n_j\}$ such that 
$\hat{p}_{n_j} \to \hat{p}$ in $L^2(0, T; H^1(0,1))$ as $j \to \infty$. Accordingly, by Lemma \ref{lem7-3} we see that 
$\Gamma^{(2)} \hat{p}_{n_j} \to \Gamma^{(2)} \hat{p}$ in $W^{2,2}(0,T)$, $\{\Gamma^{(3)} \hat{p}_{n_j} \}$ is bounded  in $L^{\infty}(0, T; H^1(0,1))$ and  
$\Gamma^{(3)} \hat{p}_{n_j} \to \Gamma^{(3)} \hat{p} $ in $L^{2}(0, T; H^1(0,1))$ as $j \to \infty$. Also, Lemma \ref{lem7-4} guarantees the boundedness
 of $\{\Gamma \hat{p}_{n_j}\}$ is $S_2(T)$. By applying Aubin's compact theorem, again, we take subsequence  $\{ \hat{p}_l \} \subset  \{ \hat{p}_{n_j} \} $ and 
 $\hat{p}_* \in S_2(T)$ such that 
$\Gamma \bar{p}_{l} \to \Gamma \bar{p}_*$ weakly in $W^{1,2}(0,T; H)$  and $L^2(0,T; H^2(0,1))$, weakly* in  $L^{\infty}(0,T; H^1(0,1))$, and
in  $L^2(0,T; H^1(0,1))$ and in $C(\overline{Q(T)})$, 
$\rho(\bar{p}_{l}) \to \rho(\bar{p}_*)$ in $L^{2}(0,T; H^1(0,1))$  and weakly in $W^{1,2}(0,T; H)$  as $l \to \infty$, where 
$\bar{p}_l = \Gamma \hat{p}_{l} $ for $l$ and $\bar{p}_*= \Gamma \hat{p}_{*}$.  
These convergences guarantee that $\bar{p}_*$ is a solution of ($\overline{\mbox{P2}}$)($\bar{v}, s, h_0, p_0)$) on $[0,T]$, where $\bar{v} = \Gamma^{(3)} \hat{p}$ and $s = \Gamma^{(2)} \hat{p}$. 
Since the uniqueness of solutions of  ($\overline{\mbox{P2}}$)($\bar{v}, s, h_0, p_0)$ holds as mentioned in Section \ref{unique}, we obtain $\hat{p} = \hat{p}_*$, This means the continuity of the mapping $\Gamma$ in the required topology. 

\color{black}

Consequently, relying on Schauder's fixed point theorem, we conclude that (P) has a solution.   
\end{proof} 

\section{Conclusion and Outlook}

In this article,  we proposed a macro-micro (two-scale) mathematical model for describing the macroscopic swelling of a rubber foam caused by the microscopic absorption of some liquid. In a first step for simplification, we assume that the occupying domain is one-dimensional. The swelling of the domain is modeled based on the standard beam equation modified by an additional part related to the liquid pressure. 

An important feature of our approach is the macro-micro structure (or dual-porosity) of the rubber foam such that the absorption takes place on a lower length scale compared to the length scale where the mechanical deformation take place (these two length scales are inherently present in this structured material). The liquid's absorption and transport inside the material was modeled by means of a nonlinear parabolic equation derived from a Darcy's law posed in a non-cylindrical domain defined inn terms of  a macroscopic deformation (solution of the beam equation). Assuming suitable conditions, we proved the existence and uniqueness of a strong solutions to our evolution system coupling the nonlinear parabolic equation posed on the microscopic non-cylindrical domain with the beam equation posed on the macroscopic cylindrical domain.  To ensure the regularity of the non-cylindrical domain, we imposed a singularity to the elastic response function in the stress structure  appearing in the beam equation. 

The results presented here are preliminary in the sense that for more research can in principle be done either what concerns the multidimensional case, or when different rheological models describe the mechanical behavior of the foam (e.g. hyperelasticity \cite{Gurtin}, {see also \cite{Ricker}}). An important question in this context is: What is the right model to use to describe the foam swelling induced by microscopic liquid absorption? Following the guidelines of the current mathematical results, a concrete swelling experiment will be designed and numerical investigations of our macro-micro model  (P) will follow. It is of special interest to see whether the current model (or eventual variations on the same theme) can replicate the swelling behaviour of foamed rubber.

\section*{Acknowledgments} T. A. and A.M. thank the Knowledge Foundation (project nr. KK 2019-0213) for supporting financially this research.
This work is partially supported also by JSPS KAKENHI Grant Number JP19K03572. {The authors like to thank O. Gehrmamm (DIK) for his support with respect to the figures.}

\bibliographystyle{plain}
%%\bibliography{c:/datafiles/texdata/data}
%%\bibliography{/Users/aikitoyohiko/Documents/OneDrive - 日本女子大学/Documents/texdata} 

%%\bibliography{data}

\end{document}